\documentclass[12pt]{amsart}
\usepackage{amssymb,float}
\usepackage[colorlinks=true]{hyperref}
\usepackage[all]{xy}
\usepackage[toc,page,title,titletoc,header]{appendix}
\usepackage{xcolor}
\usepackage{tikz-cd}
\usepackage{graphicx}
\usepackage{epigraph}

\UseRawInputEncoding
\begin{document}
	\pdfoutput=1
	\theoremstyle{plain}
	\newtheorem{thm}{Theorem}[section]
	\newtheorem*{thm1}{Theorem 1}

	\newtheorem*{thmM}{Main Theorem}
	\newtheorem*{thmA}{Theorem A}
	\newtheorem*{thm2}{Theorem 2}
	\newtheorem{lemma}[thm]{Lemma}
	\newtheorem{lem}[thm]{Lemma}
	\newtheorem{cor}[thm]{Corollary}
	\newtheorem{pro}[thm]{Proposition}
	\newtheorem{prop}[thm]{Proposition}
	\newtheorem{variant}[thm]{Variant}
	\theoremstyle{definition}
	\newtheorem{notations}[thm]{Notations}
	\newtheorem{rem}[thm]{Remark}
	\newtheorem{rmk}[thm]{Remark}
	\newtheorem{rmks}[thm]{Remarks}
	\newtheorem{defi}[thm]{Definition}
	\newtheorem{exe}[thm]{Example}
	\newtheorem{claim}[thm]{Claim}
	\newtheorem{ass}[thm]{Assumption}
	\newtheorem{prodefi}[thm]{Proposition-Definition}
	\newtheorem{que}[thm]{Question}
	\newtheorem{con}[thm]{Conjecture}

	\newtheorem{exa}[thm]{Example}
	\newtheorem*{assa}{Assumption A}
	\newtheorem*{algstate}{Algebraic form of Theorem \ref{thmstattrainv}}
	
	\newtheorem*{dmlcon}{Dynamical Mordell-Lang Conjecture}
	\newtheorem*{condml}{Dynamical Mordell-Lang Conjecture}
	\newtheorem*{congb}{Geometric Bogomolov Conjecture}
	\newtheorem*{congdaocurve}{Dynamical Andr\'e-Oort Conjecture for curves}
	
	\newtheorem*{pdd}{P(d)}
	\newtheorem*{bfd}{BF(d)}

	\newtheorem*{probreal}{Realization problems}
	\numberwithin{equation}{section}
	\newcounter{elno}                
	\def\points{\list
		{\hss\llap{\upshape{(\roman{elno})}}}{\usecounter{elno}}}
	\let\endpoints=\endlist
	\newcommand{\SH}{\rm SH}
	\newcommand{\Cov}{\rm Cov}
	\newcommand{\Tan}{\rm Tan}
	\newcommand{\res}{\rm res}
	\newcommand{\Om}{\Omega}
	\newcommand{\om}{\omega}
	\newcommand{\La}{\Lambda}
	\newcommand{\la}{\lambda}
	\newcommand{\mc}{\mathcal}
	\newcommand{\mb}{\mathbb}
	\newcommand{\surj}{\twoheadrightarrow}
	\newcommand{\inj}{\hookrightarrow}
	\newcommand{\zar}{{\rm zar}}
	\newcommand{\Exc}{{\rm Exc}}
	\newcommand{\an}{{\rm an}}
	\newcommand{\red}{{\rm red}}
	\newcommand{\codim}{{\rm codim}}
	\newcommand{\Supp}{{\rm Supp\;}}
		\newcommand{\Leb}{{\rm Leb}}
	\newcommand{\rank}{{\rm rank}}
	\newcommand{\Ker}{{\rm Ker \ }}
	\newcommand{\Pic}{{\rm Pic}}
	\newcommand{\Der}{{\rm Der}}
	\newcommand{\Div}{{\rm Div}}
	\newcommand{\Hom}{{\rm Hom}}
	\newcommand{\Corr}{{\rm Corr}}
	\newcommand{\im}{{\rm im}}
	\newcommand{\Spec}{{\rm Spec \,}}
	\newcommand{\Nef}{{\rm Nef \,}}
	\newcommand{\Frac}{{\rm Frac \,}}
	\newcommand{\Sing}{{\rm Sing}}
	\newcommand{\sing}{{\rm sing}}
	\newcommand{\reg}{{\rm reg}}
	\newcommand{\Char}{{\rm char\,}}
	\newcommand{\Tr}{{\rm Tr}}
	\newcommand{\ord}{{\rm ord}}
	\newcommand{\bif}{{\rm bif}}
	\newcommand{\AS}{{\rm AS}}
	\newcommand{\FS}{{\rm FS}}
	\newcommand{\CE}{{\rm CE}}
	\newcommand{\PCE}{{\rm PCE}}
	\newcommand{\WR}{{\rm WR}}
	\newcommand{\PR}{{\rm PR}}
	\newcommand{\TCE}{{\rm TCE}}
	\newcommand{\diam}{{\rm diam\,}}
	\newcommand{\id}{{\rm id}}
	\newcommand{\NE}{{\rm NE}}
	\newcommand{\Gal}{{\rm Gal}}
	\newcommand{\Min}{{\rm Min \ }}
	\newcommand{\Hol}{{\rm Hol \ }}

	\newcommand{\Max}{{\rm Max \ }}
	\newcommand{\Alb}{{\rm Alb}\,}
	\newcommand{\Aff}{{\rm Aff}\,}
	\newcommand{\GL}{{\rm GL}\,}        
	\newcommand{\PGL}{{\rm PGL}\,}
	\newcommand{\Bir}{{\rm Bir}}
	\newcommand{\Aut}{{\rm Aut}}
	\newcommand{\End}{{\rm End}}
	\newcommand{\Per}{{\rm Per}\,}
	\newcommand{\Preper}{{\rm Preper}\,}
	\newcommand{\ie}{{\it i.e.\/},\ }
	\newcommand{\niso}{\not\cong}
	\newcommand{\nin}{\not\in}
	\newcommand{\soplus}[1]{\stackrel{#1}{\oplus}}
	\newcommand{\by}[1]{\stackrel{#1}{\rightarrow}}
	\newcommand{\longby}[1]{\stackrel{#1}{\longrightarrow}}
	\newcommand{\vlongby}[1]{\stackrel{#1}{\mbox{\large{$\longrightarrow$}}}}
	\newcommand{\ldownarrow}{\mbox{\Large{\Large{$\downarrow$}}}}
	\newcommand{\lsearrow}{\mbox{\Large{$\searrow$}}}
	\renewcommand{\d}{\stackrel{\mbox{\scriptsize{$\bullet$}}}{}}
	\newcommand{\dlog}{{\rm dlog}\,}    
	\newcommand{\longto}{\longrightarrow}
	\newcommand{\vlongto}{\mbox{{\Large{$\longto$}}}}
	\newcommand{\limdir}[1]{{\displaystyle{\mathop{\rm lim}_{\buildrel\longrightarrow\over{#1}}}}\,}
	\newcommand{\liminv}[1]{{\displaystyle{\mathop{\rm lim}_{\buildrel\longleftarrow\over{#1}}}}\,}
	\newcommand{\norm}[1]{\mbox{$\parallel{#1}\parallel$}}
	\newcommand{\boxtensor}{{\Box\kern-9.03pt\raise1.42pt\hbox{$\times$}}}
	\newcommand{\into}{\hookrightarrow}
	\newcommand{\image}{{\rm image}\,}
	\newcommand{\Lie}{{\rm Lie}\,}      
	\newcommand{\CM}{\rm CM}
	\newcommand{\sext}{\mbox{${\mathcal E}xt\,$}}  
	\newcommand{\shom}{\mbox{${\mathcal H}om\,$}}  
	\newcommand{\coker}{{\rm coker}\,}  
	\newcommand{\sm}{{\rm sm}}
	\newcommand{\pgcd}{\text{pgcd}}
	\newcommand{\trd}{\text{tr.d.}}
	\newcommand{\tensor}{\otimes}
	\newcommand{\hotimes}{\hat{\otimes}}
	\newcommand{\crit}{{\rm crit}}

	\newcommand{\CH}{{\rm CH}}
	\newcommand{\tr}{{\rm tr}}
	\newcommand{\e}{\rm SH}
	
	\renewcommand{\iff}{\mbox{ $\Longleftrightarrow$ }}
	\newcommand{\supp}{{\rm supp}\,}
	\newcommand{\ext}[1]{\stackrel{#1}{\wedge}}
	\newcommand{\onto}{\mbox{$\,\>>>\hspace{-.5cm}\to\hspace{.15cm}$}}
	\newcommand{\propsubset}
	{\mbox{$\textstyle{
				\subseteq_{\kern-5pt\raise-1pt\hbox{\mbox{\tiny{$/$}}}}}$}}
	\newcommand{\sA}{{\mathcal A}}
	\newcommand{\sB}{{\mathcal B}}
	\newcommand{\sC}{{\mathcal C}}
	\newcommand{\sD}{{\mathcal D}}
	\newcommand{\sE}{{\mathcal E}}
	\newcommand{\sF}{{\mathcal F}}
	\newcommand{\sG}{{\mathcal G}}
	\newcommand{\sH}{{\mathcal H}}
	\newcommand{\sI}{{\mathcal I}}
	\newcommand{\sJ}{{\mathcal J}}
	\newcommand{\sK}{{\mathcal K}}
	\newcommand{\sL}{{\mathcal L}}
	\newcommand{\sM}{{\mathcal M}}
	\newcommand{\sN}{{\mathcal N}}
	\newcommand{\sO}{{\mathcal O}}
	\newcommand{\sP}{{\mathcal P}}
	\newcommand{\sQ}{{\mathcal Q}}
	\newcommand{\sR}{{\mathcal R}}
	\newcommand{\sS}{{\mathcal S}}
	\newcommand{\sT}{{\mathcal T}}
	\newcommand{\sU}{{\mathcal U}}
	\newcommand{\sV}{{\mathcal V}}
	\newcommand{\sW}{{\mathcal W}}
	\newcommand{\sX}{{\mathcal X}}
	\newcommand{\sY}{{\mathcal Y}}
	\newcommand{\sZ}{{\mathcal Z}}
	\newcommand{\A}{{\mathbb A}}
	\newcommand{\B}{{\mathbb B}}
	\newcommand{\C}{{\mathbb C}}
	\newcommand{\D}{{\mathbb D}}
	\newcommand{\E}{{\mathbb E}}
	\newcommand{\F}{{\mathbb F}}
	\newcommand{\G}{{\mathbb G}}
	\newcommand{\HH}{{\mathbb H}}
	\newcommand{\LL}{{\mathbb L}}
	\newcommand{\J}{{\mathbb J}}
	\newcommand{\M}{{\mathbb M}}
	\newcommand{\N}{{\mathbb N}}
	\renewcommand{\P}{{\mathbb P}}
	\newcommand{\Q}{{\mathbb Q}}
	\newcommand{\R}{{\mathbb R}}
	\newcommand{\T}{{\mathbb T}}
	\newcommand{\U}{{\mathbb U}}
	\newcommand{\V}{{\mathbb V}}
	\newcommand{\W}{{\mathbb W}}
	\newcommand{\X}{{\mathbb X}}
	\newcommand{\Y}{{\mathbb Y}}
	\newcommand{\Z}{{\mathbb Z}}
	\newcommand{\bk}{{\mathbf{k}}}
	
	\newcommand{\bp}{{\mathbf{p}}}
	\newcommand{\ep}{\varepsilon}
	\newcommand{\bbk}{{\overline{\mathbf{k}}}}
	\newcommand{\Fix}{\mathrm{Fix}}
	
	\newcommand{\tor}{{\mathrm{tor}}}
	\renewcommand{\div}{{\mathrm{div}}}
	
	\newcommand{\trdeg}{{\mathrm{trdeg}}}
	\newcommand{\Stab}{{\mathrm{Stab}}}
	
	\newcommand{\OK}{{\overline{K}}}
	\newcommand{\ok}{{\overline{k}}}
	
	\newcommand{\cf}{{\color{red} [c.f. ?]}}
	\newcommand{\jy}{\color{red} jy:}

	\title[]{DAO for curves}
	
\author{Zhuchao Ji}

\address{Institute for Theoretical Sciences, Westlake University, Hangzhou 310030, China}

\email{jizhuchao@westlake.edu.cn}

\author{Junyi Xie}


\address{Beijing International Center for Mathematical Research, Peking University, Beijing 100871, China}

\email{xiejunyi@bicmr.pku.edu.cn}

	
	\date{\today}

	\bibliographystyle{alpha}
	
	\maketitle
\begin{abstract}
	We prove the Dynamical Andr\'e-Oort (DAO) conjecture proposed by Baker and DeMarco for families of rational maps parameterized by an algebraic curve. In fact, we prove a stronger result, which is a Bogomolov type generalization of DAO for curves. 
	\end{abstract}
	\tableofcontents
	
	\epigraph{\emph{Dao that can be daoed is not the true Dao; Name that can be named is not the true Name.} }{Lao Zi, {\emph{Dao De Jing}}, 400 BC}
	
	\section{Introduction}
	\subsection{Statement of the main results}
	A rational map on $\P^1$ over $\C$ is called {\em postcritically finite} (PCF) if its critical orbits are finite. PCF maps play a fundamental role in complex dynamics and arithmetic dynamics, since the dynamical behavior of critical points usually reflect the general dynamical behavior of a rational map. A consequence of Thurston's rigidity theorem \cite{Douady1993} shows that PCF maps are defined over $\overline{\Q}$ in the moduli space of rational maps of fixed degree, except for the well-understood one-parameter family of flexible Latt\`es maps.  Moreover, PCF maps are Zariski dense \cite[Theorem A]{de2018dynamical}, and form a set of bounded Weil height after excluding the flexible Latt\`es family \cite[Theorem 1.1]{benedetto2014attracting}.  It was suggested by Silverman \cite{Silverman2012} that PCF points in the moduli space of rational maps  play a role analogous to that played by CM points in Shimura varieties.
	\medskip
	\par One may consider PCF maps as ``special points" in the moduli space. It is natural to ask what is the distribution of these special points.  This leads to the following conjecture \cite[Conjecture 1.1]{demarco2018critical} concerning PCF points and critical orbit relations, which is the curve case of the Dynamical Andr\'e-Oort conjecture
proposed by Baker and DeMarco in \cite{baker2013special}.
	\begin{con}[Dynamical Andr\'e-Oort conjecture for curves]\label{condao}
Let $(f_t)_{t\in \La}$ be  a non-isotrivial algebraic family of rational maps with degree $d\geq 2$,  parametrized by an algebraic curve $\La$ over $\C$. Then the following are equivalent:
\begin{points}
	\item There are infinitely many $t\in \La$ such that $f_t$ is PCF;
	\item The family has at most one independent critical orbit.
\end{points}
\end{con}
\par We need to explain the meaning of  ``independent critical orbit''.  
Set $k:=\C(\La)$.
The geometric generic fiber of the family $(f_t)_{t\in \La}$ is a rational function $f_{\overline{k}}:\P^1_{\overline{k}}\to \P^1_{\overline{k}}$.  
Following DeMarco \cite{demarco2018critical},  a pair $a,b\in \P^1(\overline{k})$  is called {\em dynamically related } if there is an algebraic curve $V\subseteq \P^1_{\overline{k}}\times \P^1_{\overline{k}}$ such that  $(a, b)\in V$  and $V$ is preperiodic by the product map $f_{\overline{k}}\times f_{\overline{k}}:\P^1_{\overline{k}}\times \P^1_{\overline{k}}\to \P^1_{\overline{k}}\times \P^1_{\overline{k}}$.  
In other words,  a pair of points $a,b\in \P^1(\overline{k})$  are dynamically related if and only if the orbit of $(a,b)\in \P^1_{\overline{k}}\times \P^1_{\overline{k}}$ is not Zariski dense in $\P^1_{\overline{k}}\times \P^1_{\overline{k}}.$
\par We say that a family $(f_t)$ has at most one independent critical orbit if any pair of critical points is dynamically related. 
\medskip
\par In \cite{baker2013special}, Baker and DeMarco actually proposed  a more general conjecture, considering also the higher dimensional parameter spaces, where almost nothing is known. Conjecture \ref{condao} and its variations  also appear in  Ghioca-Hsia-Tucker \cite{ghioca2015preperiodic}, and in DeMarco \cite{demarco2016bifurcations}, \cite{de2018dynamical},\cite{demarco2018critical}. One of the motivations of Conjecture \ref{condao} comes from the analogy between PCF points and CM points and the Andr\'e-Oort Conjecture in arithmetic geometry, which was recently fully solved by  Pila, Shankar and Tsimerman \cite{pila2021canonical}.  We note that  Conjecture \ref{condao} and the Andr\'e-Oort Conjecture both fit the principle of unlikely intersections, however,  there is no overlap between them.
 
\medskip
\par In this paper, we confirm the Dynamical Andr\'e-Oort conjecture for curves.
	\begin{thm}\label{dao}
	Conjecture \ref{condao} is true.
	\end{thm}
We indeed prove a Bogomolov type generalization of Theorem \ref{dao} (c.f. Theorem \ref{thmbogodao}) and get Theorem \ref{dao} as a simple consequence.

\medskip

In a forthcoming paper \cite{jixiemultiplier2023}, using a variant of Theorem \ref{dao}, we show that a general rational map on $\P^1$ of degree $d\geq 2$ is uniquely determined by its multiplier spectrum. This affirmatively answers a question of Poonen \cite[Question 2.43]{Silverman2012}. 

\subsection{A Bogomolov type generalization of DAO}
%
%
%
Let $\La$ be an algebraic curve over $\overline{\Q}$ and
let $(f_t)_{t\in \La}$ be  a non-isotrivial algebraic family of rational maps with degree $d\geq 2$ over $\overline{\Q}.$
The critical height $h_{\crit}: \La(\overline{\Q})\to \R$ is given by $$t\in \La(\overline{\Q})\mapsto \widehat{h}_{f_t}(\sC_{f_t}),$$
where $\sC_{f_t}$ is the critical locus of $f_t$ and $\widehat{h}_{f_t}$ is the canonical height for $f_t$ with respect to. the line bundle $\sO(1)$.
It is clear that a parameter $t\in \La(\overline{\Q})$ is PCF if and only if $h_{\crit}(t)=0.$ 
\medskip

In the following result, we prove a Bogomolov type generalization of Theorem \ref{dao} in the sense that we may replace the  PCF parameters to small height parameters.
\begin{thm}\label{thmbogodao}Let $\La$ be an algebraic curve over $\overline{\Q}$ and
let $(f_t)_{t\in \La}$ be  a non-isotrivial algebraic family of rational maps with degree $d\geq 2$ over $\overline{\Q}.$
Then the following are equivalent:
\begin{points}
	\item There are infinitely many $t\in \La$ such that $f_t$ is PCF;
	\item The family has at most one independent critical orbit;
	\item For every $\ep>0$, the set $\{t\in \La(\overline{\Q})|\,\, h_{\crit}(t)<\ep\}$ is infinite.
\end{points}
\end{thm}

Theorem \ref{dao} is a direct consequence of  Theorem \ref{thmbogodao}.

\rem Keep the notations of Theorem \ref{thmbogodao}.
Indeed our proof of Theorem \ref{thmbogodao} shows a stronger result.
Assume further that $f$ has exactly $2d-2$ marked critical points $c_1,\dots, c_{2d-2}$.
Then we may replace (iii) by the following weaker assumption:
\begin{points}
\item[(iii')]For every $i,j\in \{1,\dots, 2d-2\}$ and every $\ep>0$, the set $\{t\in \La(\overline{\Q})|\,\, \hat{h}_{f_t}(c_i)+\hat{h}_{f_t}(c_j)<\ep\}$ is infinite.
\end{points}
\endrem

\medskip

\subsection{Previous results}\label{sectionhistory}
 In \cite{baker2013special}, Baker-DeMarco proved Conjecture \ref{condao} for families of polynomials parameterized by the affine line $\A^1$ with  coefficients that were polynomial
in $t$. 
\par Since the fundamental work of Baker-DeMarco \cite{baker2013special}, plenty of works have been devoted to proving special cases of Conjecture \ref{condao} and its variations. Most progress is made in the setting that the families are given by polynomials. See Ghioca-Hsia-Tucker \cite{ghioca2013preperiodic}, Ghioca-Krieger-Nguyen \cite{ghioca2016case}, Ghioca-Krieger-Nguyen-Ye \cite{ghioca2017dynamical}, Favre-Gauthier \cite{favre2018classification} \cite{favre2022arithmetic}, and Ghioca-Ye \cite{ghioca2018dynamical}.  Among these results, a remarkable work of Favre and Gauthier \cite{favre2022arithmetic} confirmed Conjecture \ref{condao} for families of polynomials. 
\par In the case that the family is not given by polynomials, DeMarco-Wang-Ye \cite{de2015bifurcation} proved Conjecture \ref{condao} for some dynamical meaningful algebraic curves in the moduli space of quadratic rational maps. Ghioca-Hsia-Tucker  \cite{ghioca2015preperiodic} proved a weak version of Conjecture \ref{condao} for families of rational maps given by $f_t(z)=g(z)+t$, $t\in\C$, where $g\in \overline{\Q}(z)$ is of $\deg g\geq 3$ with a super-attracting fixed point at $\infty$. 

\medskip

For the Bogomolov type generalization, as far as we know, Theorem \ref{thmbogodao} is the first result.

\subsection{Strategy in the previous works for families of polynomials}
Before giving a sketch of our proof of Theorem \ref{thmbogodao} (hence Theorem \ref{dao}),  
we first recall the previous strategy for Conjecture \ref{condao}.
As mentioned in Section \ref{sectionhistory}, except for a few special cases, all progress is made in the setting that the families are given by polynomials. 
These progresses roughly follow the line of arguments devised in the original paper of Baker and DeMarco \cite{baker2013special}. 
This strategy culminates in the proof of Conjecture \ref{condao} for all one-dimensional families of polynomials by Favre and Gauthier \cite{favre2022arithmetic}.
The strategy is as follows:

\medskip

The direction that (ii) implies (i) is not hard. So we only need to show that (i) implies (ii).
It is easy to reduce to the case where $f$ is defined over some number field $K$ and has exactly $2d-2$ marked critical points $c_1,\dots, c_{2d-2}$ counted with multiplicity.

\medskip

The first step is to show that the bifurcation measure $\mu_{f,c_i}$ for $c_i, i=1,\dots,2d-2$ are proportional to the bifurcation measure $\mu_{\bif}$ via some arithmetic equidistribution theorems. The application of various equidistribution theorems is one of the most successful ideas in arithmetic dynamics, which backs to the works of Ullmo \cite{Ullmo1998} and  Zhang \cite{Zhang1998}, in where they solved the Bogomolov Conjecture. It was first introduced to study Conjecture \ref{condao}
in Baker-DeMarco's fundamental work \cite{baker2013special}.

\medskip

When the family $f$ is given by polynomials, for $i=1,\dots, 2d-2$, there is a canonical Green function $g_{f,c_i}$ on $\La(\C)$ such that $g_{f,c_i}\geq 0, \Delta g_{f,c_i}=\mu_{f,c_i}$ and $g_{f,c_i}|_{\supp\, \mu_{f,c_i}}=0$. The second step is to show that these Green functions  $g_{f,c_i}$ are proportional.
This step is easy when $\La=\A^1$, but hard in general.

\medskip

The third step is to construct an algebraic relation between $c_i$ and $c_j$ via the B\"ottcher coordinates. One may express the Green functions  $g_{f,c_i}$ using B\"ottcher coordinates. From the fact that  $g_{f,c_i}$ are proportional, one gets an analytic relation between $c_i$ and $c_j$. In the end, one shows that this analytic relation is indeed algebraic. 
This step is highly non-trivial. In \cite{baker2013special}, it was obtained by an explicit computation using properties of B\"ottcher coordinates. In \cite{favre2022arithmetic}, it relies further on an algebraization theorem for adelic series proved by the second-named author  \cite{Xie2015ring}.

\medskip

In practice, the argument could be more complicated. For example, in  each step, one needs to work on all places of $K$, not only on one archimedean place.

The second and the third steps strongly rely on the additional assumption that the families are given by polynomials for several reasons: the canonical Green functions, the B\"ottcher coordinates, the algebraization theorem, etc. 

\subsection{Sketch of our proofs}
Theorem \ref{dao} is implied by Theorem \ref{thmbogodao}.
Indeed, by Thurston's rigidity theorem for PCF maps \cite{Douady1993}, we easily reduce to the case where the family $f: \La\times \P^1\to \La\times \P^1$ over $\La$ is defined over $\overline{\Q}.$ Then Theorem \ref{dao} becomes the equivalence of (i) and (ii) in  Theorem \ref{thmbogodao}.

\medskip

For Theorem \ref{thmbogodao},
the direction that (ii) implies (i) was proved by DeMarco \cite[Section 6.4]{demarco2016bifurcations} and 
the direction that (i) implies (iii) is trivial.
We only need to show that (iii) implies (ii). 
We may assume that $f$ has exactly $2d-2$ marked critical points $c_1,\dots, c_{2d-2}$ counted with multiplicity. Assume for the sake of contradiction that $c_1$ and $c_2$ are  not dynamically related.

\medskip

\subsubsection{Structure of the proof}
There are four steps in our proof.  As in the previous works, our first step is to show the equidistribution of parameters of small height. 
We then get an additional condition that for every $c_i$,  $\mu_{f,c_i}$ is proportional to the bifurcation measure $\mu_{\bif}$ and we only need to get a contradiction under this additional assumption.  
We do this in the next three steps using a new strategy. Our basic idea is to work on general points with respect to $\mu_{\bif}$, which is motivated by the previous work \cite{jixielocal2022} of the authors on the local rigidity of Julia sets. 

In Step 2, we show  a selection  of  conditions are satisfied for $\mu_{\bif}$-a.e. point $t\in \La(\C)$, such that for a good parameter satisfies these conditions, we can construct the similarities in Step 3, moreover  we can finally get a contradiction in Step 4.

In Step 3, we construct similarities between the phase space and the parameter space. Such similarities are known in some cases for prerepelling parameters (which is a countable set). 
The novelty of our result is to get the similarities at $\mu_{\bif}$-a.e. point $t\in \La(\C)$. This leads to many difficulties coming  from the non-uniformly hyperbolic phenomenon.   

In Step 4, we  construct local symmetries of maximal entropy measures from the similarities constructed in Step 3,  and a contradiction comes from  these symmetries and an arithmetic condition that we selected in Step 2.

\subsubsection{Step 1: Equidistribution}
\par Dujardin and Favre \cite[Theorem 2.5]{Dujardin2008} (and 
DeMarco \cite{demarco2016bifurcations}) showed that the bifurcation measure $\mu_{f,c_i}$ for $c_i$ is non-zero if and only if $c_i$ is not preperiodic. In this case, $c_i$ is called active.
Since $c_1$ and $c_2$ are  not dynamically related, both of them are active.
The equidistribution theorem for small points of Yuan and Zhang \cite[Theorem 6.2.3]{yuan2021} implies that for every active $c_i$,  $\mu_{f,c_i}$ is proportional to the bifurcation measure $\mu_{\bif}$. 
This result is based on their  recent theory of adelic line bundles on quasi-projective varieties.
Before Yuan-Zhang' theory, this step was usually non-trivial in the previous works as in \cite{baker2013special} and  \cite{favre2022arithmetic}.
The quidistribution theorem is reviewed in Section \ref{2}.

\subsubsection{Step 2: Parameter exclusion}

In this step, we show  a selection  of  conditions are satisfied for $\mu_{\bif}$-a.e. point $t\in \La(\C)$, such that for a good parameter satisfies these conditions, we can construct the similarities in Step 3, moreover,  we can finally get a contradiction in Step 4.

\medskip

Since aside from the flexible Latt\`es locus, the exceptional maps \footnote{As in \cite[Section 1.1]{ji2023homoclinic}, we call $g$ \emph{exceptional} if it is a Latt\`es map or semiconjugates to a monomial map. } are isolated in the moduli space, it is easy to show that 
\begin{points}
\item[(1)] Let  $(f_t)_{t\in \La}$  be an algebraic family of rational maps over an algebraic curve $\La$,  then $\mu_{\bif}$-a.e. point  $t\in \La(\C)$ is non-exceptional. 
\end{points}
Let
$\Corr(\P^1)^{f_t}_*$ be the set of $f_t\times f_t$-invariant Zariski closed subsets $\Gamma_t\subseteq \P^1\times \P^1$ of pure dimension $1$ such that both $\pi_1|_{\Gamma_t}$ and $\pi_2|_{\Gamma_t}$ are finite.
Let $\Corr^{\flat}(\P^1_{\La})^{f}_*$ be the set of $f\times_{\La}f$-invariant Zariski closed subsets $\Gamma\subseteq \La\times(\P^1\times \P^1)$ which is flat over $\La$ and whose generic fiber is in $\Corr(\P^1_{\eta})^{f_\eta}_*$, where $\eta$ is the generic point of $\La$. In general, a correspondence $\Gamma_t\in \Corr(\P^1)^{f_t}_*$ may not be contained in any correspondence in $\Corr^{\flat}(\P^1_{\La})^{f}_*$.
On the other hand, it is the case if $t$ is \emph{transcendental} in the sense of \cite{xie2023partial} (c.f. Proposition \ref{protransco}). 
Moreover, $\Corr^{\flat}(\P^1_{\La})^{f}_*$ is countable (c.f. Proposition \ref{corfcorcountable}).
It is easy to show that 
\begin{points}
\item[(2)]Let  $(f_t)_{t\in \La}$  be an algebraic family of rational maps over an algebraic curve $\La$,  then $\mu_{\bif}$-a.e. point  $t\in \La(\C)$ is transcendental. 
\end{points}
For every $\Gamma\in \Corr^{\flat}(\P^1_{\La})^{f}_*$, we introduce a condition  $\FS(\Gamma_t)$ with respect to the pair $c_1,c_2$ for every $t\in \La(\C)$.
Roughly speaking, this condition means that in most of the time $n\geq 0$, $(f_t^n(c_1(t)), f^n_t(c_2(t)))$ is not too close to $\Gamma.$
We show that 
\begin{points}
\item[(3)] Let  $(f_t)_{t\in \La}$  be an algebraic family of rational maps over an algebraic curve $\La$,  assume moreover  that  (iii) in Theorem \ref{thmbogodao} is satisfied, then $\mu_{\bif}$-a.e. point  $t\in \La(\C)$ satisfies  the  $\FS(\Gamma_t)$ condition for every $\Gamma\in \Corr^{\flat}(\P^1_{\La})^{f}_*$. 
\end{points}
To prove (3), we introduce another condition $\AS(\Gamma_t)$ which implies $\FS(\Gamma_t).$ To show that $\AS(\Gamma_t)$ is satisfied for $\mu_{\bif}$-a.e. point, we consider integrations with respect to $\mu_{\bif}$ having arithmetic meaning, and the aim is to show that these integrations  are bounded. We bound these integrations
 via the arithmetic intersection theory on quasi-projecitve varieties \cite{yuan2021}. This was done in Section \ref{3}.

Next, we consider some typical non-uniformly hyperbolic conditions.
A result of  De Th\'elin-Gauthier-Vigny \cite{de2021parametric} shows that
\begin{points}
\item[(4)] Let  $(f_t)_{t\in \La}$  be an algebraic family of rational maps over an algebraic curve $\La$ and let $a$ be a marked point, then $\mu_{f,a}$-a.e. parameters satisfy the Parametric Collet-Eckmann and Marked Collet-Eckmann condition.
\end{points}
We introduce a condition $\PR(s)$ for $s>1/2$ with respect to a marked point $a$.
It means  that the orbit of $a$ could at most polynomially (with power $s$) close to the critical locus.
We show that 
\begin{points}
\item[(5)] Let  $(f_t)_{t\in \La}$  be an algebraic family of rational maps over an algebraic curve $\La$ and let $a$ be a marked point, then  for every $s>1/2$,  $\mu_{f,a}$-a.e. parameters satisfy the $\PR(s)$ condition.
\end{points}
This was done in Section \ref{4}. 
\par The proof of Condition (6)  is a combination of Condition (4),  Siegel's linearization theorem \cite[Theorem 11.4]{milnor2011dynamics} and the fact that 
the set of Liouville numbers has Hausdorff dimension $0$  \cite[Lemma C.7]{milnor2011dynamics},
\begin{points}
\item[(6)]  Let  $(f_t)_{t\in \La}$  be an algebraic family of rational maps over an algebraic curve $\La$, assume moreover that  for every active  marked critical point $c_i$,  $\mu_{f,c_i}$ is proportional to the bifurcation measure $\mu_{\bif}$, then  $\mu_{\bif}$-a.e. point $t\in \La(\C)$ satisfies the  Collet-Eckmann condition.
\end{points}
\par Condition (1), (2) and (6) are relatively easy to show, and Condition (4) can be easily deduced by the work of 
De Th\'elin-Gauthier-Vigny \cite{de2021parametric}. The major part of Step 2 is the proof of Condition (3) and (5). The proof of Condition (5) requires pluripotential theory. The proof of Condition (3) requires both pluripotential theory and arithmetic intersection theory. 


\subsubsection{Step 3: Similarity between the phase space and the parameter space}
\par To get a contradiction, we show that the conditions (1),(2),(4),(5),(6) imply the opposite of (3).
Our idea is to get similarity between the bifurcation measure $\mu_{\bif}$ on the parameter space and the maximal entropy measure $\mu_{f_t}$ on the phase space.
This can be thought  of as a generalization of Tan's work \cite{Tan1990}, in where she got such a similarity at Misiurewicz points in Mandelbrot set,  and as a generalization of Gauthier's \cite[Section 3.1]{gauthier2018dynamical} and Favre-Gauthier's works \cite[Section 4.1.4]{favre2022arithmetic}, in where they got such a similarity at properly prerepelling parameters.
We show that when $t\in \La(\C)$ satisfies (4),(5),(6),  for each $c_i$ active,
there is a subset $A\subseteq \Z_{\geq 0}$ of large lower density and a sequence of positive real numbers $(\rho_n)_{n\in A}$ tending to zero, such that we can construct a 
 family of renormalization maps $h_n:\D\to \P^1, n\in A$, defined by first shrinking the parameter disk $\D$ to a small disk of  radius $\rho_n$, then use $f^n$ to iterate the graph of $c_i$ over this small disk and projects to the phase space $\P^1(\C)$.  We show that this family is normal and no subsequences of $h_n, n\in A$ tending to a constant map (c.f. Theorem \ref{renor}). 
 
Comparing with the previous results, we get similarity between phase space and parameter space not only for prerepelling parameters (which is a countable set) but also for parameters satisfying Topological Collet-Eckamnn condition and Polynomial Recurrence condition  (which is a set of full $\mu_{\bif}$ measure).  
We believe that our result has an independent interest in complex dynamics.

In the previous works for prerepelling parameters, the key point is that the orbits of the marked points have a uniform distance from the critical locus for all but finitely many terms.
This is not true in our case. For this reason, we introduce the following new strategy.

Our proof is divided into two parts. 
In the first part, we work only on the phase space. 
We select the ``good time set" $A\subseteq \Z_{\geq 0}$ of large lower density. For each good time $n\in A$, we construct $n$ maps from certain fixed simply connected domain to $\P^1(\C)$. 
Roughly speaking, the goodness of $n$ means that the above maps have a uniformly bounded number of critical points. Then we need to  study the distortions of such maps which are  non-injective in general (c.f. Section \ref{plough 2}).  To describe  the distortion of  perhaps non-injective holomorphic maps, we introduce the concepts of upper and proper lower radius (c.f. Section \ref{plough 1}).
Comparing with the usual lower radius of the image, the advantage of the proper one is the stability under small perturbations.

In the second part, we use a binding argument to get the renormalization maps from the maps defined above and show that this family is normal and no subsequence of $h_n, n\in A$ tending to a constant map  (c.f. Section \ref{harvest}). In particular, we decide the rescaling factors $\rho_n, n\in A$ in this process.

%

\medskip

We also show that $\mu_{f_0}$ can be read from $\mu_{\bif}$ via the family $h_n, n\in A$ (c.f. Proposition \ref{measure}) and $\rho_n, n\in A$ can be read from $\mu_{\bif}$  up to equivalence (c.f. Proposition \ref{proresunique}).  Step 3 was done in Sections \ref{plough 1},  \ref{plough 2} and \ref{harvest}.

\subsubsection{Step 4:  Conclusion via local symmetries of maximal entropy measures}
\par We have constructed a family of renormalization maps $h_n:\D\to \P^1, n\in A$.  Since $A$ has a large lower density, after taking an intersection, we may assume that the set $A$ for $c_1$ and $c_2$ are the same.
After suitable adjustments of $\sH_a:=\{h_{a,n}, n\in A\}$ and $\sH_b:=\{h_{b,n}, n\in A\}$, we show that they form an  asymptotic symmetry, which basically means that 
every limit of $h_{a,n}\times h_{b,n}$ produces a symmetry of $\mu_{f_t}$. 
Applying an argument based on \cite[Theorem 1.7]{jixielocal2022}, we show that (3) is not true. This concludes the proof, see Section \ref{6}. 
Since we may assume that $f_t$ is Collet-Eckmann, in this last step we may replace \cite[Theorem 1.7]{jixielocal2022} by 
combing \cite[Theorem A]{dujardin2022two} with \cite[Corollary 3.2]{dujardin2022two}.

\subsection*{Acknowledgement}
We would like to thank Xinyi Yuan for discussions on the theory of adelic line bundles on quasi-projective varieties and for his help for the proofs of 
Proposition \ref{prointegboundbyint} and Proposition \ref{prointergrow}. We also would like to thank Thomas Gauthier and Laura DeMarco for helpful comments on the first version of this paper.
\par The first-named author would like to thank Beijing International Center for Mathematical Research in Peking University for the invitation. The second-named author Junyi Xie is supported by NSFC Grant (No.12271007).

	
	\section{Equidistribution of small parameters}\label{2}

\subsection{Family of rational maps}
	For $d\geq 1,$ let $\text{Rat}_d(\C)$ be the space of degree $d$ endomorphisms on $\P^1(\C)$.  
It is a smooth quasi-projective variety of dimension $2d+1$ \cite{Silverman2012}. 
The group $\PGL_2(\C)= \Aut(\P^1(\C))$ acts on $\text{Rat}_d(\C)$ by conjugacy. The geometric quotient 
$$\sM_d(\C):=\text{Rat}_d(\C)/\PGL_2(\C)$$ is the (coarse) \emph{moduli space} of endomorphisms  of degree $d$ \cite{Silverman2012}.
The moduli space $\sM_d(\C)=\Spec (\sO(\text{Rat}_d(\C)))^{\PGL_2(\C)}$ is an affine variety of dimension $2d-2$ \cite[Theorem 4.36(c)]{Silverman2007}.
Let $\Psi: \text{Rat}_d(\C)\to \sM_d(\C)$ be the quotient morphism. 
One note that, $\text{Rat}_d(\C)$, $\sM_d(\C)$ and $\Psi$ are defined over $\Q.$

	\begin{defi}
	A (one-dimensional) {\em holomorphic family of rational maps} is a holomorphic map 
		\begin{align}\label{family}
		f:\La\times\P^1(\C)&\to \La\times\P^1(\C), \\
		(t,z)&\mapsto (t,f_t(z)),  \notag
	\end{align}where $\La$ is a Riemann surface and $f_t:\P^1(\C)\to \P^1(\C)$ is a rational map of degree $d\geq 2$.

\par A holomorphic family $f$ is called {\em algebraic} if $\La$ is a smooth algebraic curve over $\C$ and the morphism $f:\La\times\P^1(\C)\to \La\times\P^1(\C)$ is algebraic.
Moreover, we say that $f$ is an algebraic family over a subfield $K$ of $\C$ if both $\Lambda$ and $f: \La\times\P^1\to \La\times\P^1$ are defined over $K.$
In other words, give an algebraic family $f$ on a smooth algebraic curve $\La$ over $\C$ is  equivalent to give an algebraic morphism $\phi_f: t\mapsto f_t\in \text{Rat}_d.$
Moreover $f$ is defined over $K$ if $\La$ and $\phi_f$ are defined over $K.$ We say that $f$ is \emph{non-isotrivial} if $\Psi\circ\phi_f$ is not a constant map.
	\end{defi}

\par 
 \par Let $\pi_1:\La\times\P^1(\C)\to \La$ and $\pi_2:\La\times\P^1(\C)\to \P^1(\C)$ be the canonical projections. Let $\omega_{\P^1}$ be the Fubini-Study form on $\P^1(\C)$, and let  $\omega_\La$ be a fixed K\"ahler form  on $\La$ with $\int_{\La} \omega_{\La}=1$. Let  $\omega_1:=\pi_1^\ast (\omega_{\La})$ and $\omega_2:=\pi_2^\ast (\omega_{\P^1})$.  The {\em relative Green current}\footnote{In \cite{Gauthier2019}, it is called ``fibered Green current".} of $f$ is defined by 
    \begin{equation*}
    T_f:=\lim_{n\to+\infty} d^{-n} (f^n)^\ast (\omega_2)=\omega_2+dd^c g,
    \end{equation*}
where $g$ is a H\"older continuous quasi-p.s.h. function \cite[Lemma 1.19]{dinh2010dynamics}.
\par For every $t_0\in \La$, we have $T_f\wedge [t=t_0]=\mu_{f_{t_0}}$, where $\mu_{f_{t_0}}$ is the maximal entropy measure of $f_{t_0}$.
	\par A {\em marked point} $a$ is a holomorphic map $a:\La\to \P^1(\C)$. The {\em bifurcation measure} of the pair $(f,a)$ is defined by
	\begin{equation*}
		\mu_{f,a}:=(\pi_1)_\ast (T_f\wedge [\Gamma_a])=a^*T_f,
	\end{equation*}
where $\Gamma_a$ is the graph of $a$. When the family $f$ is algebraic, the marked point $a$ is said to be \emph{algebraic} if the map $a:\La\to \P^1(\C)$ is algebraic. When $a$ is algebraic, $\mu_{f,a}$ has finite mass. See Gauthier-Vigny \cite[Proposition 13 (1)]{Gauthier2019}, where they proved the finiteness of the mass of bifurcation currents' for general algebraic families of polarized dynamical systems. 
\par To simplify the notation, for an algebraic family $f$, a marked point $a$  is always assumed to be algebraic in the whole paper.

\par A marked point $a$ is called {\em active} if $\mu_{f,a}$ does not vanish. Otherwise, it is called {\em passive}. 
 \par A {\em marked critical point} $c$ is a marked point such that $c(t)$ is a critical point of $f_t$ for each $t\in \La$.  Let $f$ be an algebraic family of rational maps as in (\ref{family}) with marked critical points $(c_i)_{1\leq i\leq 2d-2}.$ We define the {\em bifurcation measure} of $f$ by $$\mu_{\bif}:=\sum_{i=1}^{2d-2}\mu_{f,c_i}.$$
 Since each $c_i$ is algebraic, $\mu_{\bif}$ has finite mass. 
 \medskip
 \par The following theorem is useful.
 \begin{thm}[DeMarco, \cite{demarco2016bifurcations}]\label{stable}
  Let $f$ be a holomorphic family of rational maps as in (\ref{family}) and $a$ be  a marked point.  Then the following are equivalent:
  \begin{points}
  \item $a$ is passive;
  \item $a$ is  stable, i.e. the family of maps $\left\{t\to f_t^n(a(t))\right\}_{n\geq 1}$  forms a normal family.
  \end{points}
\par If moreover, $f$ is a non-isotrivial algebraic family, then the above two conditions are equivalent to that $a$ is  preperiodic.
 \end{thm}
When $a$ is a marked critical point, the last statement was proved in an earlier article \cite[Theorem 2.5]{Dujardin2008}.

\medskip

\subsection{Equidistribution}
The following deep result about equidistribution of preperiodic points  is a direct consequence of  Yuan-Zhang \cite[Theorem 6.2.3]{yuan2021}. Note that for an algebraic family of rational maps as in (\ref{family}) defined over a number field $K$, the canonical height of a preperiodic point is equal to $0$. See also Gauthier \cite[Theorem 3]{gauthier2021good}.
\begin{thm}\label{equi}
Let $f$ be an algebraic family of rational maps as in (\ref{family}) and $a$ be  an active marked point, all defined over a number field $K$. Let $t_n\in \La(\overline{\Q})$ be an infinite sequence of distinct points such that $\widehat{h}_{f_t}(a(t_n))\to 0$ as $n\to 0$. Then we have 
$$\frac{1}{|\Gal(t_n)|}\sum_{t\in \Gal(t_n)}\delta_t\to\frac{\mu_{f,a}}{\mu_{f,a}(\La)},$$
when $n\to+\infty$, where $\Gal(t_n)$ is the Galois orbit of $t_n$. 
\end{thm}
\medskip
\par The following is an important corollary of the equidistribution of parameters of small heights.

\begin{cor}\label{corequi}
 Let $f$ be an algebraic  family of rational maps as in (\ref{family}) with marked critical points $(c_i)_{1\leq i\leq 2d-2}$.  Assume that $f$ is defined over $\overline{\Q}$ and 
 there is an infinite sequence $t_n, n\geq 0$ of distinct points in $\La(\overline{\Q})$
 such that $h_{\crit}(t_n)\to 0$, then for $c_i, c_j$ being active marked critical points, we have
 $$\frac{\mu_{f,c_i}}{\mu_{f,c_i}(\La)}=\frac{\mu_{f,c_j}}{\mu_{f,c_j}(\La)}.$$
\end{cor}
\begin{proof}
We may assume that $f$ has at least two active marked critical points. Otherwise, the statement is trivial.
This implies that $f$ is not isotrivial and $\phi_f(\La)$ is not contained in the locus of flexible Latt\`es maps. 

\medskip

As $h_{\crit}(t_n)\to 0$, both $\widehat{h}_{f_{t_n}}(c_i(t_n))$ and $\widehat{h}_{f_{t_n}}(c_j(t_n))$ tend to zero.
By Theorem \ref{equi}, both $\frac{\mu_{f,c_i}}{\mu_{f,c_i}(\La)}$ and $\frac{\mu_{f,c_j}}{\mu_{f,c_j}(\La)}$ equal to $\frac{1}{|\Gal(t_n)|}\sum_{t\in \Gal(t_n)}\delta_t,$ which concludes the proof.
%
%
%
\end{proof}
 \medskip

\section{Frequently separated parameters}\label{3}
\subsection{Frequently separated condition}
Let $(M,d)$ be a metric space and $g: M\to M$ be a self-map.

\begin{defi}
	Let $A$ be a subset of $\Z_{\geq 0}$. The {\em asymptotic lower/upper density} of $A$ is defined by 
	\begin{equation*}
		\underline{d}(A):=\liminf_{n\to \infty} |A\cap[0,n-1]|/n,
	\end{equation*}
	and 
	\begin{equation*}
		\overline{d}(A):=\limsup_{n\to \infty} |A\cap[0,n-1]|/n.
	\end{equation*}
	If $\underline{d}(A)=\overline{d}(A)$, we set $d(A):=\underline{d}(A)=\overline{d}(A)$ and call it the  {\em asymptotic density} of $A$.
\end{defi}

\medskip

We still denote by $d$ the distance in on $M\times M$ by $$d((x_1,y_1),(x_2,y_2))=\max\{d(x_1,x_2),(y_1,y_2)\}.$$
Let $\Sigma$ be a non-empty subset of $M\times M$. 

\rem\label{remcorronefactor}We view $\Sigma$ as a correspondence on $M.$ The most typical example is the diagonal. 
For $x\in M$, we denote by $\Sigma(x):=\pi_2(\pi_1^{-1}(x))$, where $\pi_1,\pi_2$ are the first and the second projections.
When $\pi^{-1}(x)\neq \emptyset$, for every $y\in M$, we have $d((x,y), \Sigma)\geq d(y, \Sigma(x)).$ 
\endrem

\begin{defi}A pair of points $x,y\in M$ is called 
 \begin{points}
 	\item {\em Frequently separated} $\FS(\Sigma)$ for $\Sigma$, if for every $\ep>0$, there is $\delta>0$ such that 
	$$\overline{d}(\{n\geq 0|\, d((g^n(x),g^n(y)),\Sigma)\geq \delta\})>1-\ep.$$
	\item {\em Average separated}  $\AS(\Sigma)$ for $\Sigma$, if 
	$$\liminf_{n\to \infty}\frac{1}{n}\sum_{i=0}^{n-1} \max\{-\log d((g^n(x),g^n(y)),\Sigma), 0\}<+\infty.$$
	\end{points}
\end{defi}
The above conditions  depend only on the equivalence class of distance functions on $M$ and $M\times M.$
It is clear that if $\Sigma\subseteq \Sigma'$, then $\AS(\Sigma')$ (resp. $\FS(\Sigma')$) implies $\AS(\Sigma)$ (resp. $\FS(\Sigma)$).

\rem When $\Sigma$ is the diagonal $\Delta$, the $\FS(\Delta)$ condition means that in most of the time $n\geq 0$, the orbits of $x$ and $y$ are $\delta$-separated for some small $\delta>0$. The proportion of such time could tend to  $1$ when $\delta$ tends to $0.$
\endrem

\medskip

\begin{lem}\label{lemasimplyfs}The $\AS(\Sigma)$ condition implies the $\FS(\Sigma)$ condition.
\end{lem}

\proof
Assume that the $\AS(\Sigma)$ condition holds.  
Set $$\phi_n:=\max\{-\log d((g^n(x),g^n(y)),\Sigma), 0\}$$
and $$s_n:=\frac{1}{n}\sum_{i=0}^{n-1}\phi_n.$$

The $\AS(\Sigma)$ condition shows that there is $A\geq 0$ and a sequence $n_j, j\geq 0$ tending to $+\infty$ such that 
$s_n\leq A.$ For $\ep>0$, we have $$\frac{\#\{i=0,\dots,n_j-1|\,\, \phi_i\geq A/\ep\}}{n_j}\leq \ep.$$
Set $\delta:=e^{-A/\ep}$, we get $$\frac{\#\{i=0,\dots,n_j-1|\,\, d((g^i(x),g^i(y)),\Sigma)\geq \delta\}}{n_j}\geq 1-\ep,$$
which concludes the proof.
\endproof
\subsection{Correspondences for rational maps}
Let $\bk$ be a field.  Let $g: \P^1_{\bk}\to \P^1_{\bk}$ be an endomorphism of degree $d\geq 2.$

\medskip

A \emph{correspondence} of $\P^1_{\bk}$ is a non-empty Zariski closed subset of $\P^1_{\bk}\times \P^1_{\bk}.$
We denote by $\Corr(\P^1_{\bk})^g$ the set of correspondences of $\P^1_{\bk}\times \P^1_{\bk}$ which are invariant under $g\times g$ and 
$\Corr(\P^1_{\bk})^g_*$ the subset of $\Corr(\P^1_{\bk})^g$ consisting of those $\Gamma$ which are of pure dimension $1$ and such that both $\pi_1|_{\Gamma}$ and $\pi_2|_{\Gamma}$ are finite.


\begin{lem}\label{lemincordes}Let $K$ be a subfield of $\bk$ such that $g$ is defined over $K$. Then for every $\Gamma\in \Corr(\P^1_{\bk})^g_*$, there is $\Gamma_K\in \Corr(\P^1_{K})^g_*$ such that $\Gamma\subseteq \Gamma_K\otimes_K\bk.$
\end{lem}
\proof When $\bk=\overline{K},$ we may define $\Gamma_K$ to be the image of $\Gamma$ under the natural morphism $(\P^1\times \P^1)_\bk\to (\P^1\times \P^1)_K$.
Now we may assume that both $\bk$ and $K$ are algebraically closed.
Then every preperiodic point is defined over $K$.

For every $\Gamma\in \Corr(\P^1_{\bk})^g_*$, we have $$P_{\Gamma}:=\Gamma\cap \pi^{-1}(\Preper(g))=\Gamma\cap \pi_2^{-1}(\Preper(g))\subseteq \Preper(g)\times \Preper(g),$$
where $\pi_1,\pi_2$ are the first and the second projections.
Since $\Preper(g)$ is Zariski dense in $\P^1_{\bk}$ and $\pi_1|_{\Gamma}$ is surjective, $P_{\Gamma}$ is Zariski dense in $\Gamma.$ Since every points in $P_{\Gamma}$ are defined over $K$, $\Gamma$ is defined over $K$, which concludes the proof.
\endproof

\begin{lem}\label{lemgincoun}The set $\Corr(\P^1_{\bk})^g$  is countable. 
\end{lem}
\proof
We may assume that $\bk$ is algebraically closed.
There is an algebraically closed subfield $K$ of $\bk$, such that $K$ has finite transcendence degree and $g$ is defined over $K.$
Since $\Corr(\P^1_{K})^g_*$ is countable, we conclude the proof by Lemma \ref{lemincordes} and the fact that 
$\Corr(\P^1_{\bk})^g\setminus \Corr(\P^1_{\bk})^g_*$  is countable.
\endproof

\medskip

Let $f: \La\times \P^1\to \La\times \P^1$ be an algebraic family of rational maps as in (\ref{family}).
A \emph{flat family of correspondences} over $\La$ is a closed subset $\Gamma\subseteq (\P^1\times \P^1)\times \La$ which is flat over $\La.$
Let $\eta$ be the generic point of $\La$. The map $\Gamma\mapsto \Gamma_{\eta}$ gives a bijection between flat family of correspondences over $\La$
and correspondences of the generic fiber  $\P^1_{\eta}.$ Moreover $\Gamma$ is $f\times_{\La} f$-invariant if and only if $\Gamma_{\eta}$ is $f_{\eta}\times f_{\eta}$-invariant.
Hence, by Lemma \ref{lemgincoun}, we get the following result.

\medskip

\begin{cor}\label{corfcorcountable} The set $\Corr^{\flat}(\P^1_{\La})^f$ of $f\times_{\La} f$-invariant flat family of correspondences over $\La$ is countable.
\end{cor}
Denote by $\Corr^{\flat}(\P^1_{\La})^f_*$ the set of $\Gamma\in \Corr^{\flat}(\P^1_{\La})^f$  whose generic fiber is in $\Corr^{\flat}(\P^1_{\eta})^{f_{\eta}}_*.$

\medskip
\subsubsection{Transcendental points}
The notion of transcendental points was introduced in \cite{xie2023partial}. Let $a_i, i=1,\dots.m$ be marked points.
Let $L$ be an algebraically closed subfield of $\C$ such that $\La$, $f$ and $a_i, i=1,\dots,m$ are defined over $L.$
Then there is a variety $\La_0$ over $L$, a morphism $F:\La_0\times \P^1\to \P^1$ and marked points $a'_1,\dots, a'_m$ such that $\La=\La_0\otimes_{L}\C$, $f=F\otimes_{L}\C$ and $a_i=a'_i\otimes_L\C.$
A point $b\in \La(\C)=\La_0(\C)$ is called a \emph{transcendental point} for $\La_0/L$
if the image of 
$b:\Spec \C\to \La_0$
is the generic point of $\La_0$.  
In other words, $b\in \La(\C)=\La_0(\C)$ is transcendental if and only if it is not in $\La_0(L)$.
\begin{rem}\label{remdescou}
We can always assume $L$ to have a finite transcendence degree. In this case, $L$ is countable, hence 
 $\La_0(L)$ is countable. So all but countably many points in $\La(\C)$ are transcendental for $\La_0/L$. 
 In particular, if $f$ is defined over $\overline{\Q}$, and $a_i$ are marked critical points, we can take $L$ to be $\overline{\Q}.$
\end{rem}

\begin{rem}\label{remdescouae}
 Since $\mu_{f,a_i}$ has continuous potential, it does not have atoms.
So if $L$ have finite transcendence degree, the for $\mu_{f,a_i}$-a.e. $t\in \La(\C)$, $t$ is transcendental with respect to $\La_0/L$. 
\end{rem}

\begin{pro}\label{protransco}Let $b\in \La(\C)$ be a transcendental point with respect to $\La_0/L$. Then for every $\Gamma_b\in \Corr(\P^1)_*^{f_{b}}$,
there is $\Gamma'\in  \Corr^{\flat}(\P^1_{\La})_*^f$ such that $\Gamma'_b:=\Gamma'\cap \pi_1^{-1}(b)$ contains $\Gamma_b.$ 
Moreover, for every $n\geq 0$ and $i,j\in \{1,\dots,m\}$, if 
$(f_b^n(a_i(b)),f_b^n(a_j(b)))\in \Gamma_b$, then $(f^n(a_i),f^n(a_j))\in \Gamma'.$
\end{pro}

\proof
Since the image of $b:\Spec \C\to \La_0$ is the generic point of $\La_0$,
$f_b$ and $(a_i)_b, i=1,\dots,m$  are the base change of $F_{K}, (a_i')_K, i=1,\dots,m$ via the natural morphism 
$$K:=L(\La_0)\hookrightarrow \C$$ defined by $b.$
By Lemma \ref{lemincordes}, for every $\Gamma_b\in \Corr(\P^1)_*^{f_{b}}$, there is $\Gamma'_K\in \Corr(\P^1_K)_*^{F_K}$
such that $\Gamma_b\subseteq \Gamma'_K\otimes_K\C.$  
Via the natural bijection between $\Corr(\P^1_K)_*^{F_K}$ and $\Corr^{\flat}(\P^1_{\La_0})_*^{F}$,
$\Gamma'_K$ is the generic fiber of $\Gamma'_{L}\in \Corr^{\flat}(\P^1_{\La_0})_*^{F}.$
Set $\Gamma':=\Gamma_L'\otimes_L \C\in \Corr^{\flat}(\P^1_{\La})_*^f.$
Then we get $\Gamma'_K\otimes_K\C=\Gamma'_b.$ For every $n\geq 0$ and $i,j\in \{1,\dots,m\}$, if 
$(f_b^n(a_i(b)),f_b^n(a_j(b)))\in \Gamma_b$, then $(F_K^n((a_i')_K),F_K^n((a_j')_K))\in \Gamma_K'$, hence
$(f^n(a_i),f^n(a_j))\in \Gamma'.$ This concludes the proof.
\endproof

\subsection{The $\AS(\Sigma)$ condition for families of rational maps}

\begin{thm}\label{thmavergdistance}
 Let $f$ be an algebraic family of rational maps as in (\ref{family}). Let $a,b$ be active marked points.
 Let $V$ be a Zariski closed subset of $\La\times (\P^1\times \P^1)$.
Assume that $f$, $a,b$ and $V$ are defined over $\overline{\Q}$; for every $n\geq 0$, the image $\Gamma_n$ of $p_n:=(f^n(a),f^n(b)): \La\to  \La\times (\P^1\times \P^1)$ is not contained in 
$V;$ and there is a sequence of distinct points $t_i\in  \La(\C), i\geq 0$ such that both $\widehat{h}_{f_{t_i}}(a(t_i))$ and $\widehat{h}_{f_{t_i}}(b(t_i))$ tend to zero as $i\to \infty.$
 Then for $\mu_{f,a}$-a.e. $t$ in $\La(\C)$, the pair $a(t), b(t)$ satisfies the $\AS(V_t)$ condition for $f_t.$
\end{thm}
 
 \medskip
 
 \rem
 By Theorem \ref{stable} and \cite[Proposition 13]{Gauthier2019}, $\mu_{f,a}$ and $\mu_{f,b}$ are of finite non-zero mass.
 Moreover, by Theorem \ref{equi}, $\mu_{f,a}$ and $\mu_{f,b}$ are proportional.
 \endrem
 
  \medskip
 Combing Corollary \ref{corfcorcountable}, Remark \ref{remdescou}, Remark \ref{remdescouae} with Proposition \ref{protransco},  we get the following result.
\begin{cor}\label{corinvariantas}Let $f,a,b$ as in Theorem \ref{thmavergdistance}.  
Assume further that $a,b$ are not dynamically related.
Then for $\mu_{f,a}$-a.e. $t$ in $\La(\C)$, we have that for every 
$\Gamma_t\in \Corr(\P^1)^{f_t}_*$,  the pair $a(t), b(t)$ satisfies the $\AS(\Gamma_t)$ condition for $f_t$.
\end{cor}

\proof[Proof of Theorem \ref{thmavergdistance}]
Let $P_i: \La\times (\P^1\times \P^1)\to \La\times \P^1, i=1,2$ be the morphism defined by $(t,x_1,x_2)\mapsto (t, x_i).$
Let $\pi_1:=\La\times (\P^1\times \P^1)\to \La$ be the first projection. Let $T:=P_1^*T_f+P_2^*T_f.$ 
Set $g:=f\times_{\La}f.$ 
Since $\mu_{f,a}$ does not have atomic point, we may assume that $V$ is of dimension $2$ and is flat over $\La.$
Let $B$ be a smooth projective curve containing $\La$ as an open subset.
Then $W:=\overline{V}$ is a Cartier divisor of $B\times (\P^1\times \P^1).$

Let $g_W$ be a Green function on for $W$ i.e. $g_W$ is a continuous function on $B\times (\P^1\times \P^1)(\C)$
such that 
for every point $y\in W(\C)$, there is an open neighborhood $U$ of $y$ such that 
$g_{W}=-\log |h|+O(1)$ where $W$ is defined by $h=0$ in $U.$ We may assume that $g_W\geq 0.$
The following lemma is the key to the proof.
\begin{lem}\label{lemintegarith}There is $C>0$ such that for every $n\geq 0,$
$$\int g_W \,T\wedge [\Gamma_n]\leq Cd^n.$$
\end{lem}
Since the proof of Lemma \ref{lemintegarith} is based on the arithmetic intersection theory,
we postpone it to the end of this section. We indeed prove a more general result which implies Lemma \ref{lemintegarith} as a direct consequence.
For every $t\in \La(\C)$, $g_{W_t}:=g_W|_{\pi_1^{-1}(t)}$ is a Green function of $W_t=V_t.$
Then we have $$Cd^n\geq \int g_W \,P^*_1T\wedge \Gamma_n=d^n\int g_{W_t}(p_n(t)))\mu_{f,a}.$$
Hence $\int g_{W_t}(p_n(t)))\mu_{f,a}\leq C.$
Set $\phi_n(t):=\max\{-\log d(p_n(t),V), 0\}.$
There is $C_1\geq 1$ such that  
$$\max\{-\log d(p_n(t),V), 0\}\leq C_1(g_{W_t}(p_n(t)))+1).$$
So there is $C_2>0$ such that for every $n\geq 0,$
$\int \phi_n(t)\mu_{f,a}\leq C_2.$
Set $s_n:=\frac{1}{n}\sum_{i=0}^{n-1}\phi_n$, we have $\int s_n\mu_{f,a}\leq C_2.$
According to  Fatou's lemma,
$$\int (\liminf_{n\to \infty}s_n) \mu_{f,a}\leq \liminf_{n\to \infty}\int s_n \mu_{f,a}\leq C_2.$$
So for $\mu_{f,a}-a.e.$ $t\in \La(\C)$, we have $\liminf\limits_{n\to \infty}s_n(t)<+\infty$ which concludes the proof.
\endproof

\subsubsection{Chern-Levine-Nirenberg inequality}
Recall the Chern-Levine-Nirenberg inequality from pluripotential theory, see \cite[Theorem A.31]{dinh2010dynamics}, \cite{Chern1969} and \cite[(3.3)]{Demailly}.
\begin{thm}\label{CLN}
	Let $(X,\omega)$ be a Hermitian manifold. Let $S$ be a positive closed current of bi-dimension $(1,1)$ on $X$. Let $u$ be  a  locally bounded p.s.h. function on $X$ and $K$ a compact subset of $X$. Then there is a constant $c=c(X,K)>0$ such that if $v$ is p.s.h. on $X$, then
	\begin{equation*}
		\int_K |v|\; dd^c u \wedge S\leq c \left(\int_X |v| \;S\wedge \omega \right) \|u\|_{L^\infty(X)}.
	\end{equation*}
\end{thm}

Recall that a function $v$ is called quasi-p.s.h. if there is $C>0$ such that $C\omega+dd^c v\geq 0.$
Then we have the following two variants of the Chern-Levine-Nirenberg inequality.
\begin{cor}\label{corglobelclnv}
Let $(X,\omega)$ be a K\"ahler manifold. Let $S$ be a positive closed current of bi-dimension $(1,1)$ on $X$. 
Let $v$ be a quasi-p.s.h. function on $X$. Then there is a constant $c''=c''(X,K,v)>0$ such that if $u$ be  a  locally bounded p.s.h. function on $X$, then
	\begin{equation*}
		\int_K |v|\; dd^cu \wedge S\leq c' \|u\|_{L^\infty(X)}.
	\end{equation*}
\end{cor}
\proof
Assume that $C\omega+dd^c v\geq 0.$
For every $x\in K$, there is a compact neighborhood $Z_x$ of $x$ and an open neighborhood $U_x$ of $Z_x$ such that   $\omega|_{U_x}=dd^cw_x$ where $w_x$ is a bounded continuous p.s.h. functions. 
There is a constant $B_x\geq 0$, such that $-B_x\leq w_x\leq B_x$ on $U_x.$
There is a finite set $F\subseteq K$ such that $K\subseteq \cup_{x\in F}Z_x.$
Set $v_x':=v+Cw_x+B_x$ and $v_x'':=v+Cw_x-B_x$. Both $v_x'$ and $v_x$ are locally bounded p.s.h. on $U_x$ and 
$v_x'\leq v\leq v_x''.$ By Theorem \ref{CLN}, we have 
\begin{align*}
\int_K |v|\; dd^cu \wedge S\leq& \sum_{x\in F}\int_{K\cap Z_x} |v|\; dd^cu \wedge S\\
\leq & \sum_{x\in F}\int_{K\cap Z_x} (|v_x'|+ |v_x''|)\; dd^cu \wedge S\\
\leq & \left(\sum_{x\in F}c_{K\cap Z_x, U_x}\int_{U_x} (|v_x'|+ |v_x''|)\; \omega \wedge S\right)\|u\|_{L^\infty(X)},
\end{align*}
which concludes the proof.
\endproof

Since every closed positive $(1,1)$-current having  continuous potential locally takes form $dd^c u$ for some  continuous p.s.h. function $u$, 
the proof of the following result is similar to the proof of Corollary \ref{corglobelclnv}.
\begin{cor}\label{corglobelcln}
Let $(X,\omega)$ be a Hermitian manifold. Let $S$ be a positive closed current of bi-dimension $(1,1)$ on $X$. 
Let $T$ be a closed positive $(1,1)$-current having continuous potential. Then there is a constant $c'=c'(X,K,T)>0$ such that if $v$ is a p.s.h. on $X$, then
	\begin{equation*}
		\int_K |v|\; T \wedge S\leq c' \left(\int_X |v| \;S\wedge \omega \right).
	\end{equation*}
\end{cor}

\subsubsection{Integrations via arithmetic intersection theory}
This section is based on the theory of adelic line bundles on quasi-projective varieties developed by Yuan and Zhang in \cite{yuan2021}.
We follow their notations and terminologies. 
 
 \medskip
 
 Let $S$ be a smooth curve over a number field $K.$ 
 Let $S'$ be a smooth projective curve containing $S$ as an open subset.
Let $\pi':X'\to S'$ be a projective model of $\pi: X\to S.$ Set $D_S:=S'\setminus S$ and $\overline{D_S}=(D_S, g_{D_s})$ be an ample arithmetic divisor associated with $D_S,$ where $g_{D_s}$ is a continuous Green function.
Let $D:=\pi'^{-1}(D_S)$ and $\overline{D}:=\pi^*\overline{D_S}$. Write $\overline{D}=(D, g_{D})$ where $g_D=\pi^*g_{D_s}.$
 
 \medskip
 
 Let $(X,g,L)$ be a \emph{polarized dynamical system} over $S$ i.e.
 \begin{points}
 \item[(1)] $\pi:X\to S$ is a projective and flat scheme over $S;$
 \item[(2)] $g: X\to X$ is an endomorphism over $S$;
 \item[(3)] $L\in \Pic(X)$ is a line bundle, relatively ample over $S$, such that $g^*L=qL$ for some integer $q>1.$
 \end{points}

 \medskip
 
In \cite[Theorem 6.1.1]{yuan2021}, by Tate's limiting argument, Yuan and Zhang constructed an adelic line bundle $\overline{L}_g\in \widehat{\Pic}(X)$ extending $L$ such that $g^*\overline{L}_g=q\overline{L}_g.$
Moreover,  it is  strongly nef in the following sense:
There is a sequence of nef adelic line bundles $\overline{L_n},n\geq 1$ which are defined over some projective models $\pi_n: X_n'\to S'$ of $\pi: X\to S,$ and positive numbers $\ep_n, n\geq 0$ tending to $0$,
such that $\overline{L_g}-\overline{L_n}$ is represented by an arithmetic divisor $\overline{D_n}=(D_n, g_{D_n})\in \widehat{\Pic}(X)_{int}$ with $$-\ep_n\overline{D}\leq \overline{D_n}\leq \ep_n \overline{D}.$$
We may assume that for every $n\geq 0$, $X_n'$ dominates $X'$.
For every archimedean place  $v\in \sM_K$, $c_1(\overline{L}_g)_v$  is a positive $(1,1)$-current on $X_v^{\an}$ having following properties 
 \begin{points}
 \item $c_1(\overline{L}_g)_v$ has continuous potentials;
 \item for every $t\in \La$, $c_1(\overline{L}_g)_v|_{X_t}$ is a Green current for $g_t$;
 \item $\int c_1(\overline{L}_g)_v^{\dim X}=0.$
 \end{points}
 
 \medskip
 
Let $p:S\to X$ be a section of $\pi$ and let $\Gamma_p$ be its image. 
Let $\overline{A}$ be an ample adelic line bundle defined on $X',$ let $A\in \Pic(X')$ be the line bundle associated with $\overline{A}$.

\begin{pro}\label{prointegboundbyint}Let $s$ be a small section of $A\in \Pic(X')$ i.e. for every $v\in M_K$ and $x\in X'^{\an}_v$, $|s(x)|_v\leq 1.$
Assume that $\Gamma_p\not\subseteq \div(s).$ Then for every archimedean place  $v\in \sM_K$, we have 
$$\int g_{\div(s),v} \, c_1(\overline{L}_g)_v\wedge [\Gamma_p]\leq \overline{A}\cdot \overline{L}_g\cdot \Gamma_p,$$
where $g_{\div(s),v}(x)=-\log |s(x)|_v$ is the Green function.
\end{pro}

\medskip

\proof[Proof of Proposition \ref{prointegboundbyint}]
Let $Y$ be any compact subset of $X_v^{\an}$. Pick $U_Y$ an open neighborhood of $Y$ such that $U_Y\subset\subset X_v^{\an}.$ 

Since $g_{\div(s),v}$ is a quasi-p.s.h. function,  by Corollary \ref{corglobelclnv}, there is a constant $c=c(Y,U_Y, g_{\div(s),v})$ such that for every $n\geq 0$, we have 
$$\int_Y g_{\div(s),v}dd^cg_{D_n,v}\wedge [\Gamma_p]\leq c\|g_{D_n,v}\|_{L^{\infty}(U_Y)}\leq \ep_nc\|g_{D,v}\|_{L^{\infty}(U_Y)}.$$
Then for every $n\geq 0$, we have 
\begin{equation}\label{equintapp}
\begin{split}&\int_Y g_{\div(s),v}c_1(\overline{L}_g)_v\wedge [\Gamma_p]\\=&\int_Y g_{\div(s),v} c_1(\overline{L_n})_v\wedge [\Gamma_p]
+\int_Y g_{\div(s),v}dd^cg_{D_n,v}\wedge [\Gamma_p]\\
\leq &\int_Y g_{\div(s),v} c_1(\overline{L_n})_v\wedge [\Gamma_p]
+\ep_nc\|g_{D,v}\|_{L^{\infty}(U_Y)}.
\end{split}
\end{equation}

Since $X_n'$ dominates $X'$, we may view $s$ as a section on $X_n.$
Let $\Gamma_{p,n}$ be the Zariski closure of $\Gamma_p$ in $X_n.$
By \cite[Theorem 1.4]{Chambert-Loir2009} (see also \cite[Page 1161]{Yuan2017}), 
we have 
$$\overline{A}\cdot \overline{L_n}\cdot \Gamma_{p,n}=\overline{L_n}\cdot (\Gamma_{p,n}\cdot \div s)+\sum_{w\in \sM_k}\int g_{\div(s),w} c_1(\overline{L_n})_w\wedge [\Gamma_p].$$
Since each term on the right hand side is positive,
we get 
\begin{equation}\label{equboundln}
\int_Y g_{\div(s),v} c_1(\overline{L_n})_v\wedge [\Gamma_p]\leq \int g_{\div(s),v} c_1(\overline{L_n})_v\wedge [\Gamma_p]\leq \overline{A}\cdot \overline{L_n}\cdot \Gamma_{p,n}.
\end{equation}
The definition of $\overline{A}\cdot \overline{L}_g\cdot \Gamma_{p}$ shows that
\begin{equation}\label{equdefiadeint}\overline{A}\cdot \overline{L}_g\cdot \Gamma_{p}=\lim_{n\to \infty}\overline{A}\cdot \overline{L_n}\cdot \Gamma_{p,n}.
\end{equation}
As $\ep_n\to 0$, we conclude the proof by (\ref{equintapp}), (\ref{equboundln}) and (\ref{equdefiadeint}).
\endproof

 Set $\Gamma_n:=\Gamma_{g^n(p)}.$
 \begin{pro}\label{prointergrow}Assume that $\overline{L_g}^2\cdot \Gamma_p=0$, then  for every  ample adelic line bundle $\overline{A}$ defined on $X'$, there is $C>0$
 such that for every $n\geq 0$,
  $$\overline{A}\cdot \overline{L_g}\cdot \Gamma_{n}\leq Cq^n,$$
  where $q$ is the integer in the definition (3) of the polarized dynamical system  $(X,g, L)$.
 \end{pro}
 \proof
 For every $n\geq 0$, by projection formula, we have  
 \begin{equation}\label{queintergamn}\overline{L_g}^2\cdot \Gamma_n=\overline{L_g}^2\cdot g_*^n(\Gamma_p)= g^{*n}(\overline{L_g})^2\cdot \Gamma_p=q^{2n}\overline{L_g}^2\cdot \Gamma_p=0.
 \end{equation}

 As $L$ is relatively ample over $S$, after replacing $X'$ by some projective model $X''$ which dominates $X'$ and $X'_0$, $L$ by some multiple of it and $\overline{A}$ by $\overline{A}+\overline{A'}$ for some ample adelic line bundle $\overline{A'}$ on $X''$, we may assume that $L$ extends to a line bundle on $X'$.  
 We now view $L$ and $L_0$ as line bundles on $X'.$

 There is an integer $m_1>0$  and an ample line bundle $M$ on $S'$ such that 
 $E:=m_1L+M-A$ is ample.  Pick a suitable metric on $E$, we get an ample adelic line bundle $\overline{E}$ extending $E.$
 After replacing $\overline{A}$ by $\overline{A}+\overline{E}$, we may  assume that $A=m_1L+M$ on $X'.$
 Since $m_1L_0-A$ is trivial on the generic fiber of $\pi': X'\to S'$, $m_1\overline{L_0}-\overline{A}$ is represented by  
 an arithmetic divisor $\overline{F}$ on $X'$ such that $\pi(\supp F)\neq S'.$   Then there is an ample arithmetic divisor $\overline{Q}$ on $S'$  such that 
 $$-\pi^*\overline{Q}\leq \overline{F}\leq \pi^*\overline{Q}.$$

Set $R:=\ep_0\overline{D_S}+\overline{Q}.$
 By (\ref{queintergamn}), we get
 \begin{equation*}
 \begin{split}
 \overline{A}\cdot \overline{L}_g\cdot \Gamma_{n}=&(m_1\overline{L_g}-m_1\overline{D_0}-\overline{F})\cdot\overline{L}_g\cdot \Gamma_{n}\\
 \leq &m_1\overline{L_g}^2\cdot \Gamma_{n}+\pi^*\overline{R}\cdot\overline{L}_g\cdot \Gamma_{n}\\
 =&\pi^*\overline{R}\cdot\overline{L}_g\cdot \Gamma_{n}\\
 =&\pi^*\overline{R}\cdot\overline{L}_g\cdot g^n_*(\Gamma_p)\\
 =&(g^n)^*\pi^*\overline{R}\cdot (g^n)^*\overline{L}_g\cdot \Gamma_p\\
 =&q^n(\pi\circ g^n)^*\overline{R}\cdot\overline{L}_g\cdot \Gamma_p\\
 =&q^n(\pi^*\overline{R}\cdot\overline{L}_g\cdot \Gamma_p),
 \end{split}
 \end{equation*}
 which concludes the proof.
  \endproof
 
 Combining Proposition \ref{prointegboundbyint} and Proposition \ref{prointergrow}, we get the following result.
 \begin{cor}\label{corgvdn}Let $V$ be a Zariski closed subset of $X'.$ 
 Let $v\in \sM_K$ be an archimedean place. Let $g_V$ be a Green function of $V.$
Assume that $\overline{L_g}^2\cdot \Gamma_p=0$. Then there is $C>0$ such that, for every 
$n\geq 0$, if $\Gamma_n$ is not contained in $V,$ we have
 $$\int g_{V}\, c_1(\overline{L}_g)_v\wedge [\Gamma_n]\leq  Cq^n,$$
 where $q$ is the integer in the definition (3) of the polarized dynamical system  $(X,g, L)$.
 \end{cor}
 \proof
 Set $N:=\{n\geq 0|\,\, \Gamma_n\not\subseteq V\}.$
Let $I_V$ be the ideal sheaf of $V$ in $X'.$
We fix an ample adelic line bundle  $\overline{A}$ defined on $X',$ and we denote $A\in \Pic(X')$ to be the line bundle associated with $\overline{A}$. By replacing $\overline{A}$ by a suitable multiple, we may assume that $A\otimes I_V$ is generated by global sections.
Then there are sections $s_1,\dots, s_m$ of $A$ such that $\cap_{i=1}^m\div(s_i)=V.$
For every $n\in N$, there is $i_n\in \{1,\dots,m\}$ such that $\Gamma_n\not\subseteq \div(s_{i_n}).$
After modifying the metric of $A$, we may assume that for every $i=1,\dots,m$, $s_i$ is small for $\overline{A}$ and 
$g_{\div(s_i),v}\geq g_V.$
By Proposition \ref{prointegboundbyint} and Proposition \ref{prointergrow}, there is $C>0$ such that for every $n\in N$, we have
$$\int g_{V}\, c_1(\overline{L}_g)_v\wedge [\Gamma_n]\leq  \int g_{\div(s_{i_n})}\, c_1(\overline{L}_g)_v\wedge [\Gamma_n]   \leq \overline{A}\cdot \overline{L}_g\cdot \Gamma_n\leq Cq^n,$$
which concludes the proof.
 \endproof
 
\proof[Proof of Lemma \ref{lemintegarith}]
Since $\La, f,a,b$ and $V$ are defined over $\overline{\Q}$, there is 
a number field $K$, a smooth curve $S$ over $K$ and a polarized endomorphism $h: \P^1_S\to \P^1_S$ over $S$ of degree $d$, a closed subset $W$ of $X$ and a section $p: S\to X:=(\P^1\times \P^1)\times S$
of $\pi: X\to S$ such that $\La$, $f$, $(a,b): \La\to (\P^1\times \P^1)\times \La$ and $V$ are the base change of $S$, $h$ and $p$ via an embedding $K\hookrightarrow \C$ defined by an archimedean place $v\in \sM_K.$
Let $S'$ be a smooth projective curve containing $S$ as a Zariski closed subset. We still denote by $W$ by its Zariski closure in $X':=(\P^1\times \P^1)\times S'.$
Let $L:=\pi_1^*O(1)+\pi_2^*O(1)$ where $\pi_i: \P^1\times \P^1\to \P^1$ is the $i$-th projection. Set $g:=h\times_{S'}h.$ Let $\Gamma_n$ be the image of $g^n(p).$
Then we have $T=c_1(\overline{L_g})_{v}.$
Let $\widetilde{L_g}\in \widetilde{\Pic(X)}$ be the geometric part of $\overline{L_g}.$
By \cite[Lemma 5.4.4]{yuan2021}, we have 
$$\widetilde{L_v}\cdot \Gamma_0=\int c_1(\overline{L_g})_{v}\wedge \Gamma_0=\int T\wedge \Gamma_0=(\mu_{f,a}+\mu_{f,b})(\La)>0.$$
As $\Gamma_0$ contains an infinite sequence $(a(t_i),b(t_i)), i\geq 0$ of distinct points with heights tend to $0$, the fundamental inequality \cite[Theorem 5.3.2]{yuan2021} shows that $\overline{L_v}^2\cdot \Gamma_0=0.$
We conclude the proof by Corollary \ref{corgvdn}.
\endproof

 
	\section{Typical non-uniformly hyperbolic conditions}\label{4}
In this section, we show that certain non-uniformly hyperbolic conditions are typical for the bifurcation measure. We begin with some definitions.
 Let $f$ be a holomorphic family of rational maps as in (\ref{family}) and $a$ be  a marked point.   In this section, the distance and the norm of the derivatives are computed with respect to the metrics induced by $\omega_\La$ and $\omega_{\P^1}$. 
 \par For every $n\geq 0$, let $\xi_{a,n}:\La\to \P^1(\C)$ denote the map $\xi_{a,n}(t):=f_t^n(a(t))$. 
\begin{defi}
 A parameter $t_0\in \La$ is called 
 \begin{points}
 	\item {\em Marked Collet-Eckmann} $\CE^*(\la)$ for some $\la>1$, if there exists $C>0$ and $N>0$ such that $$|df_{t_0}^n(f_{t_0}^N(a(t_0)))|\geq C\la^n$$ for every $n\geq 0$. 
 	\item  {\em Parametric Collet-Eckmann} $\PCE(\la)$ for some $\la>1$, if there exists $C>0$ such that $$\left|\frac{d\xi_{a,n}}{dt}(t_0)\right|\geq C\la^n$$ for every $n\geq 0$. 
 	\item  {\em Polynomial Recurrence} $\PR(s)$ for some $s>1/2$, if there exists an integer $N>0$ such that
 	$$ d(f_{t_0}^n(a(t_0)), \mathcal{C}_{t_0})\geq n^{-s}$$
 	for every $n\geq N$.  Where $\mathcal{C}_{t_0}$ is the critical set of $f_{t_0}$ 
\end{points}
\end{defi}
\medskip
\begin{lemma}\label{precri}
   Let $f$ be a holomorphic family of rational maps as in (\ref{family}) and $a$ be  a marked point. Assume $t\in \supp \mu_{f,a}$, then there exists $N>0$ such that the  $f_{t}$-orbit of $f_t^N(a(t))$ does not intersect $\mathcal{C}_{t}$.  
\end{lemma}
\begin{proof}
Assume by contradiction that there is no such $N$. Since the cardinality of $\mathcal{C}_{t}$ is finite, there exists $c\in \mathcal{C}_{t}$ such that $c$ is $f_t-$periodic and $a(t)$ is a preimage of $c$.  This implies $a(t)$ is contained in the super-attracting basin of the $f_t-$orbit of $c$. Since the attracting basin is stable under perturbation, it implies that $a$ is stable at $t$, contradicts to $t\in \supp \mu_{f,a}$ by Theorem \ref{stable}.
\end{proof}
\medskip
\par The following result was essentially due to De Th\'elin-Gauthier-Vigny. 
\begin{thm}[De Th\'elin-Gauthier-Vigny, \cite{de2021parametric}]\label{DGV}
   Let $f$ be an algebraic family of rational maps as in (\ref{family}) and $a$ be  a marked point. Then we have:
\begin{points}
	\item For every $1<\la<d^{1/2}$, the condition $\PCE(\la)$ is typical with respect to $\mu_{f,a}$;
\item Assume $t$ satisfies $\PCE(\la_0)$ for some $\la_0>1$, then $t$ satisfies $\CE^*(\la)$ for every $1<\la<\la_0$.
\end{points}
\end{thm}
\begin{proof}
The statement (i) was proved in \cite[Theorem 3]{de2021parametric}. In \cite[Proposition 9]{de2021parametric}, it is proved that if  $t$ satisfies $\PCE(\la_0)$ for some $\la_0>1$, then $t$ satisfies $\CE^*(\la)$ for every $1<\la<\la_0$, provided that the $f_{t}$-orbit of $a(t)$ does not intersect $\mathcal{C}_{t}$.  The condition $\PCE(\la_0)$ implies  $t\in \supp \mu_{f,a}$, by Lemma \ref{precri}, we get that  (ii) holds. 
\end{proof}

\medskip

\subsection{A transversality statement}
Let $f$ be a holomorphic family of rational maps as in (\ref{family}) and $a$ be  a marked point. A direct computation gives the following equality, relating $df_{t_0}^n(a(t_0))$ and $d\xi_{a,n}/dt(t_0)$. A proof can be found in \cite[Lemma 4.4]{astorg2019collet}. We set $F(t,z):=f_t(z)$. 
\begin{lemma}\label{relation}
 Let $t_0\in\La$ such that  the $f_{t_0}$-orbit of $a(t_0)$ does not intersect $\mathcal{C}_{t_0}$. Then for every $n\geq 1,$ we have
 \begin{equation*}
 \frac{d\xi_{a,n}}{dt}(t_0)=df_{t_0}^n(a(t_0))\left(\frac{da}{dt}(t_0)+\sum_{k=0}^{n-1}\frac{\frac{\partial F}{\partial t}(t_0,f_{t_0}^k(a(t_0)))}{df_{t_0}^{k+1}(a(t_0))}\right).
 \end{equation*}
\end{lemma}
\medskip
The following transversality statement holds for $\PCE$ parameters. 
\begin{lemma}[Transversality condition]\label{trans}
 Let $t_0\in\La$ such that  the $f_{t_0}$-orbit of $a(t_0)$ does not intersect $\mathcal{C}_{t_0}$, and  $t_0$ satisfies $\PCE(\la)$ for some $\la>1$. Then  there exists a non-zero $\gamma\in \C$ such that
  \begin{equation*}
 	\frac{da}{dt}(t_0)+\sum_{k=0}^{\infty}\frac{\frac{\partial F}{\partial t}(t_0,f_{t_0}^k(a(t_0)))}{df_{t_0}^{k+1}(a(t_0))}=\gamma.
 \end{equation*}
\end{lemma}
\begin{proof}
By Theorem \ref{DGV} (ii), $t_0$ satisfies $\CE^*(\la_1)$ for every $1<\la_1<\la$. Since 
$|\partial F/\partial t(t_0,f_{t_0}^k(a(t_0))|$ is uniformly bounded by a constant $M>0$,  the power seris in Lemma \ref{trans}  converges. It remains to show that it converges to a non-zero number.  
\par Let $\chi>0$ be the lower Lyapunov exponent of $f_{t_0}$ at $a(t_0)$. 
Let $\ep>0$ small such that $\ep<\min\;(\log \la/10,\chi)$. By the definition of lower Lyapunov exponent, there exists a constant $C_1(\ep)>0$ and a sequence  of positive integers $n_j\to +\infty$ such that 
\begin{equation}\label{3.1}
|df_{t_0}^{n_j}(a(t_0))|\leq C_1 e^{(\chi+\ep)n_j}.
\end{equation}There is also a constant $C_2(\ep)>0$ such that for every $n\geq 1$,
\begin{align}\label{3.2}
\left|\sum_{k=n+1}^{\infty}\frac{\frac{\partial F}{\partial t}(t_0,f_{t_0}^k(a(t_0)))}{df_{t_0}^{k+1}(a(t_0))}\right|&\leq \sum_{k=n+1}^{\infty}\left|\frac{\frac{\partial F}{\partial t}(t_0,f_{t_0}^k(a(t_0)))}{df_{t_0}^{k+1}(a(t_0))}\right|
\notag\\&\leq \sum_{k=n+1}^{\infty} \frac{M}{C_2e^{(\chi-\ep)k}}\notag\\
&:= C_3e^{-(\chi-\ep)n},
\end{align}
where $C_3(\ep)>0$ is a constant. 
\par Assume by contradiction that the power seris in  Lemma \ref{trans} converges to $0$. By Lemma \ref{relation}, (\ref{3.1}) and (\ref{3.2}), for each $n_j$ in (\ref{3.1}) we have
\begin{align*}
	 \left|\frac{d\xi_{a,n_j}}{dt}(t_0)\right|&=\left|df_{t_0}^{n_j}(a(t_0))\right|\left|\sum_{k=n_j+1}^{\infty}\frac{\frac{\partial F}{\partial t}(t_0,f_{t_0}^k(a(t_0)))}{df_{t_0}^{k+1}(a(t_0))}\right|
	 \\&\leq C_1 e^{(\chi+\ep)n_j}C_3e^{-(\chi-\ep)n_j}\\
	 &=C_1C_3e^{2\ep n_j},
\end{align*}
which contradicts to the fact $t_0$ satisfies $\PCE(\la)$. 
\end{proof}

\medskip

\subsection{Polynomial Recurrence parameters are typical}
In this subsection, we prove the following:
\begin{thm}\label{PR}
	Let $f$ be an algebraic family of rational maps as in (\ref{family}) and $a$ be  a marked point. Then for every $s>1/2$, $\mu_{f,a}$-a.e. point $t_0\in \La$ satisfies $\PR(s)$. 
\end{thm}
There is a finite Zariski open cover $\sW$ of $\La$, such that for every $W\in \sW$ and every  marked critical point $c$, there is an algebraic family $g_{W,c}: W\to \PGL_{2,\C}$
such that for every $t\in W(\C)$, we have $$(g_{W,c}(t))(c(t))=0\in \P^1(\C).$$
Since we only need to prove Theorem \ref{PR} for the restriction of $f$ on each $W\in \sW.$ We may assume that $\sW=\{\La\}$ and write $g_{c}$ for $g_{\La,c}$.

\medskip

For every marked critical point $c$, define an algebraic automorphism $\sigma_c: \La\times \P^1\to \La\times \P^1$ sending $(t, z)$ to $(t, (g_{c}(t))(z)).$
It is clear that $\sigma_c(\Gamma_c)=\La\times\{0\}.$ Recall that $\om_1=\pi_1^*\,\omega_\La$ and $\om_2=\pi_2^* \,\om_{\P^1}.$ Set $\omega:=\omega_1+\omega_2$. 
For every closed algebraic curve $V\subseteq \La\times \P^1$, 
if $\pi_2(V)$ is not a single point, then for every $z\in \P^1(\C)$, we have $\# (\pi_2|_V)^{-1}(z)\leq \deg_{\omega_2}V$. The equality holds for all but finitely many $z\in \P^1(\C).$ There is a constant $C_0>0$ such that for every marked critical point $c$ and every closed algebraic curve $V\subseteq \La\times \P^1$, we have 
\begin{equation}\label{eqc0}\deg_{\omega}\sigma_c(V)\leq C_0\deg_{\omega}V.
\end{equation}
Note that we always have $\deg_{\omega}V\geq 1.$
Let $a$ be a marked point, we have $\deg_{\omega_1}\Gamma_a=1$, hence $\deg_{\omega}\Gamma_a=\deg_{\omega_2}\Gamma_a+1.$ 
Since $\sigma_c(\Gamma_a)$ is the graph of the marked point $\sigma(a): t\in \La\to (g_c(t))(a(t))\in \P^1$, we have $\deg_{\omega}\sigma_c(\Gamma_a)=\deg_{\omega_2}\sigma_c(\Gamma_a)+1.$  By (\ref{eqc0}), we get 
\begin{equation}\label{equmomo2}\deg_{\omega_2}\sigma_c(\Gamma_a)\leq \deg_{\omega}\sigma_c(\Gamma_a)\leq C_0\deg_{\omega}\Gamma_a.
\end{equation}
There is $D>0$ such that for every  marked point $a$, 
\begin{equation}\label{eqheightineq}\deg_{\omega}\Gamma_{f(a)}\leq d\deg_{\omega}\Gamma_{a}+D
\end{equation}
where $d:=\deg f\geq 2$.
It follows that for every $n\geq 0,$
\begin{equation}\label{eqheightineqit}\deg_{\omega}\Gamma_{f^n(a)}\leq \deg_{\omega}\Gamma_{f^n(a)}+D\leq d^n(\deg_{\omega}\Gamma_{a}+D).
\end{equation}
We set  $$\phi(t,z):= d(z,\mathcal{C}_t)^{-1},$$
where $\mathcal{C}_t$ is the critical set of $f_t$.
\begin{lemma}\label{exp1}
	Let $K$ be compact subset of $\La$ and $0<\alpha<2$. Then there exists a constant $C:=C(K,\alpha)>0$ such that for every  marked point  $a$ which is not a marked critical point and for every $A>0$ we have
	$$\int_{K\times \P^1(\C)}\min (A,  \phi^\alpha) \;T_f\wedge [\Gamma_a]\leq C( A+ \deg_{\omega}\Gamma_a).$$
\end{lemma}
\begin{proof}
Pick an open neighborhood $U$ of $K$ which is relatively compact in $\La$.
There is a constant $B_1=B_1(U)>1$ such that for every marked critical point $c$ and $(t,z)\in \pi_1^{-1}(U)$, we have
\begin{equation}\label{equcomparedsigma}B_1^{-1}d((g_c(t))(z),0)\leq d(z, c(t))\leq B_1d((g_c(t))(z),0).
\end{equation}
There is  $\delta_0>0$ and $B_2>1$, such that for every $z\in \P^1(\C)$ with $d(z, 0)\leq \delta_0$, we have 
\begin{equation}\label{equdz0z}B_2^{-1}|z|\leq d(z, 0)\leq B_2|z|.
\end{equation}
Set 
$$\Omega_c:=\left\{(t,z):t\in U\;\text{and}\; d(z,c(t))<\delta_0\right\}.$$
Set $B_0:=B(0,B_1\delta_0)$, then we have $$\sigma_c(\Omega_c)\subseteq U\times B_0.$$
For every marked critical point $c$, set $\phi_c(t,z):=d(z, c(t))^{-1}.$ Then we have $\phi=\max_c \phi_c$ where $c$ is taken over all the $2d-2$ marked critical points.
Set $B_3:=B_1B_2.$ By (\ref{equcomparedsigma}) and (\ref{equdz0z}), for every $(t,z)\in \sigma_c(\Omega_c)$, 
$$\phi_c\circ \sigma_c^{-1}(t,z)\leq B_3|z|^{-1}.$$
Set $V_c:=\sigma_c(\Gamma_a).$
We have
\begin{flalign*}
&\int_{K\times \P^1(\C)}\min (A,  \phi^\alpha) \;T_f\wedge [\Gamma_a]\\=&\int_{K\times \P^1(\C)}\max_c\min (A,  \phi_c^\alpha) \;T_f\wedge [\Gamma_a]\\
\leq &\sum_c\int_{K\times \P^1(\C)}\min (A,  \phi_c^\alpha) \;T_f\wedge [\Gamma_a]\\
=&\sum_c\int_{(K\times \P^1(\C))\setminus \Omega_c}\min (A,  \phi_c^\alpha) \;T_f\wedge [\Gamma_a]+\\
&\sum_c\int_{\sigma_c(((K\times \P^1(\C))\cap\Omega_c)}\min (A,  \phi_c^\alpha\circ \sigma_c^{-1}) \;{\sigma_c}_*(T_f)\wedge [V_c]\\
\leq& (2d-2)\delta_0^{-\alpha}\int_{K\times \P^1(\C)}T_f\wedge [\Gamma_a]+\\
&\sum_c\int_{K\times B_0}\min (A,  B_3^{\alpha}|z|^{-\alpha}) \;{\sigma_c}_*(T_f)\wedge [V_c].
\end{flalign*}
Since $T_f$ has continuous potential, by Corollary \ref{corglobelcln}, there is a constant $c'$ depending on $K$ such that 
\begin{equation}\label{eqestoutom}\int_{K\times \P^1(\C)}T_f\wedge [\Gamma_a]\leq c'\deg_{\omega}\Gamma_a.
\end{equation}
 So we only need to show that for every marked critical point $c$, 
$$\int_{\overline{K\times B_0}}\min (A,  B_3^{\alpha}|z|^{-\alpha}) \;{\sigma_c}_*(T_f)\wedge [V_c]\leq  C_1( A+\deg_{\omega} \Gamma_a)$$
for some constant $C_1$ depending on $K.$
Since  $$\min(A,B_3^{\alpha}|z|^{-\alpha})=e^{-(\max (-\log A,\alpha \log |z|-\alpha\log B_3))}$$ is a p.s.h. function, and ${\sigma_c}_*(T_f)$ has continuous potential, by Corollary \ref{corglobelcln}, there is a constant $C'$ 
such that 
\begin{align*}
		&\int_{\overline{K\times B_0}}\min (A,  B_3^{\alpha}|z|^{-\alpha}) \;{\sigma_c}_*(T_f)\wedge [V_c]\\ \leq& C'\left(\int_{U\times 2B_0}\min (A,  B_3^{\alpha}|z|^{-\alpha})  \;  \omega\wedge [V_c]\right)
	\end{align*}
We only need to bound $\int_{U\times 2B_0}\min (A,  B_3^{\alpha}|z|^{-\alpha})  \;  \omega\wedge [V_c].$
 It suffices to show 
	\begin{equation}\label{3.4}
		\int_{U\times 2B_0}\min (A,  B_3^{\alpha}|z|^{-\alpha}) [V_c]\wedge \om_1\leq C_2 A
	\end{equation}
	and
	\begin{equation}\label{3.5}
		\int_{U\times 2B_0}\min (A,  B_3^{\alpha}|z|^{-\alpha})[V_c]\wedge \om_2\leq C_3\deg_{\om}\Gamma_a
	\end{equation}
	for some constants $C_2>0$ and $C_3:=C_3(\alpha)>0$. 

\par Since $V_c$ is a graph, we have 
	\begin{equation*}
		\int_{U\times 2B_0}\min (A,  B_3^{\alpha}|z|^{-\alpha}) [V_c]\wedge \om_1 \leq A\int_{U\times 2B_0}[V_c]\wedge \om_1=
		A\int_{U}\om_{\La},
	\end{equation*}
hence  (\ref{3.4}) is true.
On the other hand, since $a$ is not a marked critical point and the function $|z|^{-\alpha}$ is Lebesgue integrable when $0<\alpha<2$, we have 
	\begin{align*}
		\int_{U\times 2B_0}\min (A,  B_3^{\alpha}|z|^{-\alpha})[V_c]\wedge \om_2\leq & B_3^{\alpha}\int_{U\times 2B_0}|z|^{-\alpha}[V_c]\wedge \om_2\\
		\leq & B_3^{\alpha}\int_{2B_0}|z|^{-\alpha}({\pi_2}_*[V_c])\wedge \om_{\P^1}.
	\end{align*}
	By (\ref{equmomo2}), $\deg_{\om_2}V_c\leq C_0\deg_{\om}\Gamma_a$, we have
	$$\int_{2B_0}|z|^{-\alpha}({\pi_2}_*[V_c])\wedge \om_{\P^1}\leq C_0\deg_{\om}\Gamma_a\int_{2B_0}|z|^{-\alpha}\om_{\P^1},$$
	hence (\ref{3.5}) is true. This finishes the proof.
\end{proof}

\medskip
Let $a$ be a marked point.
For every $n\geq 0$, set  $\psi_{ n}(t):=\phi(t,f_t^n(a(t)))$. Recall that $\mu_{f,a}=(\pi_1)_\ast (T_f\wedge [\Gamma_a]).$ 
We note that $\mu_{f,f^n(a)}=d^n\mu_{f,a}$ for every $n\geq 0$. 
A direct corollary of Lemma \ref{exp1} and Inequality (\ref{eqheightineqit}) is the following:


\begin{cor}\label{exp2}
	Let $K$ be compact subset of $U$ and $0<\alpha<2$. Then there exists a constant $C:=C(K,\alpha)>0$ such that for every $n\geq 0,$ we have
	$$\int_{K} \min (A, \psi_n^{\alpha}) \;d\mu_{f,a}\leq C(d^{-n}A+\deg_{\om} \Gamma_a+D).$$
\end{cor}
\medskip
\begin{lemma}\label{markov}
	Let $K$ be compact subset of $U$ and $0<\alpha<2$. Then there exists a constant $C(K,\alpha)>0$ such that for every $n\geq 0$, we have
	$$\mu_{f,a}(t\in K: \psi_{n}^\alpha (t)\geq A)\leq C(d^{-n}+A^{-1}(\deg_{\om} \Gamma_a+D)).$$
\end{lemma}
\begin{proof}
	By Markov inequality  and Corollary \ref{exp2} we have
	\begin{align*}
		\mu_{f,a}(t\in K: \psi^\alpha _{ n}\geq A)&=\mu_{f,a}(t\in K: \min (A, \psi^\alpha _{ n})\geq A)\\&
		\leq A^{-1} \int_{K} \min (A, \psi_n^{\alpha}) \;d\mu_{f,a}
		\\&\leq  C(d^{-n}+A^{-1}(\deg_{\om} \Gamma_a+D)).
	\end{align*}
\end{proof}
\medskip
\proof[Proof of Theorem \ref{PR}]
 It suffices to prove for every $s>1/2$ and for every compact subset $K\subseteq \La$, there exists a set $E(K,s)\subseteq K$ satisfying $\mu_{f,a}(K\setminus E)=0$ such that every $t_0\in E$ satisfies $\PR(s)$. 
 For every $n\geq 1,$ set $$F_n:=\left\{t\in K:d(f_t^{n}(a(t)), \mathcal{C}_t)\leq n^{-s}\right\}.$$ Pick a constant  $\alpha\in (1/s, 2)$. By Lemma \ref{markov}, for $n$ large enough, there exists a constant $C:=C(K,\alpha)>0$ such that
\begin{align*}
\mu_{f,a}(F_n)&=\mu_{f,a}(t\in K: \psi_{n}(t)^{\alpha}\geq  n^{\alpha s})
\\&\leq C(d^{-n}+n^{-\alpha s}(\deg_{\om}\Gamma_a+D))
\end{align*}
This implies that $$\sum_{n=1}^\infty \mu_{f,a}(F_n)<\infty,$$
and we conclude the proof by  the Borel-Cantelli lemma.
\endproof

Combine Theorem \ref{DGV} with Theorem \ref{PR}, we have:
\begin{pro}\label{nuh}
 Let $f$ be an algebraic family of rational maps as in (\ref{family}) and $a$ be  an active marked point.  Let $1<\la<d^{1/2}$ and $s>1/2$, then $\mu_{f,a}$-a.e. $t\in \La$ satisfies $\CE^*(\la)$, $\PCE(\la)$ and $\PR(s)$,  in particular $a(t)\in \sJ(f_t).$
\end{pro}

\medskip

\section{Distortion of non-injective maps}\label{plough 1}
\begin{defi}
A rational map $g:\P^1(\C)\to \P^1(\C)$ is called {\em Topological Collet-Eckmann} $\TCE(\la)$ for some $\la>1$ if  there exists $\delta_0>0$ such that for each $n\geq 0$ and $z\in \sJ(g)$, we have 
$$\diam W_n\leq \la^{-n},$$
where $W_n$ is any connected component of $g^{-n}(B(z,\delta_0))$, and $\sJ(g)$ is the Julia set of $g$.
\end{defi}

The aim of Section \ref{plough 1}, \ref{plough 2} and \ref{harvest} is to prove the following theorem.  Recall that for a holomorphic family of rational maps $f:\D\times \P^1\to \D\times \P^1$  as in (\ref{family}) and for  a marked point $a$,  we let $\xi_{a,n}:\D\to \P^1(\C)$ be the map $\xi_{a,n}(t):=f_t^n(a(t))$. 
\begin{thm}\label{renor}
Let $f:\D\times \P^1\to \D\times \P^1$ be a holomorphic family of rational maps as in (\ref{family}) and $a$ be a marked point.  Assume $0\in \D$ satisfies
\begin{points}
	\item $\PCE(\la_0)$ for some $\la_0>1$;
	\item $\PR(s)$ for some $s>0$;
	\item $f_{0}$ is $\TCE(\la)$ for some $\la>1$;
	
\end{points} 
Then for every $\ep>0$,  there exists a subset $A\subseteq \Z_{\geq 0}$ with $\underline{d}(A)>1-\ep$ and  $0<\rho_m<1, m\in A$
 such that the family $\{h_m, m\in A\}$ of maps, 
\begin{align*}
h_m:\D&\to \P^1(\C),
\\t&\mapsto \xi_{a,m}\left( \rho_m t\right)
\end{align*}
 form a normal family, for which every limit map is non-constant.
 Moreover, we have $\rho_m\to 0$ as $m\to \infty.$
\end{thm}
\par The proof of Theorem \ref{renor} is given in Section \ref{harvest}. In Section \ref{plough 1}, \ref{plough 2}  we do some preparations. 
\medskip
\par Section \ref{plough 1} is devoted to proving some distortion properties for non-injective holomorphic maps. In this section the distances on $\P^1(\C)$ are computed with respect to the spherical metric.  Note that under the spherical metric, $\diam(\P^1(\C))=\pi.$ For every $a\in \P^1(\C)$, the unique point $a^-$ satisfying $d(a,a^-)=\pi$ is the 
antipodal point of $a$. We have $0^-=\infty$.
When $a\not\in\{0,\infty\}$, we have $a^{-}=-\overline{a}^{-1}.$

\subsection{Upper radius and proper lower radius}
Let $\Omega$ be a connected Riemann surface  and $x\in \Omega$.  
Let $h:\Omega\to \P^1(\C)$ be a holomorphic map. 
\par We say a connected open neighborhood $U$ of $h(x)$ in $\P^1(\C)$ is \emph{properly in the image of $(h,\Omega,x)$} and write $U\subseteq_p h(\Omega,x)$, if there is a connected open neighborhood $W$ of $x$ in $\Omega$ such that
$h|_{W}: W\to U$ is proper. 

\par Easy to check the following properties.
\begin{pro}\label{proproinima}
\begin{points}
\item If $U\subseteq_p h(\Omega,x)$, then for every connected open neighborhood $V$ of $h(x)$ contained in $U$, $V\subseteq_p h(\Omega,x).$
\item Let $\Omega'\subseteq \Omega$ be a connected open neighborhood of $x$, if $U\subseteq_p h(\Omega',x)$, then $U\subseteq_p h(\Omega,x).$
\end{points}
\end{pro}

The following criterion is useful.
\begin{lem}\label{lempropercriterion}
Assume that $\Omega$ is compactly contained in a Riemann surface $S$.
Assume that $h$ is not constant and extends to a neighborhood of $\overline{\Omega}.$
Let $V$ be a connected open neighborhood of $h(x).$
If there is a connected open neighborhood $W$ of $x$ and $D>0$ such that for every $y\in V$, it has exactly $D$ preimages under $h|_W$ counted with multiplicities, then 
$V\subseteq_p h(\Omega, x).$
\end{lem}
\proof
We may assume that $h$ extends to $S$ and $W=\Omega.$
Set $U:=h|_{\Omega}^{-1}(V)$ and let $U_0$ be the connected component of $U$ containing $x.$
We only need to show that $h|_{U_0}:U_0\to V$ is proper. So we only need to show that $h|_U$ is proper. 

If $h|_U$ is not proper, then there is a compact subset $K$ of $V$ such that $$h^{-1}(K)\cap \Omega \neq h^{-1}(K)\cap \overline{\Omega}.$$
Pick $z\in \partial\Omega\cap h^{-1}(K).$ Let $z_1,\dots z_s$ be the preimages of $h(z)$ under $h|_{\Omega}$ with multiplicities $m_1,\dots,m_s$.
Then $\sum_{i=1}^s m_i=D.$ Pick open neighborhoods $W_i$ of $z_i$ in $\Omega$ and $W_0$ of $z$ such that $W_i\cap W_j=\emptyset$ for $i\neq j.$
Pick $w\in W_0\cap \Omega$ sufficiently close to $z.$ Then $h(w)$ has exactly $m_i$ preimages in $W_i, i=1,\dots,s$ counted with multiplicities and has a preimage $w$ in $W_0\cap \Omega.$
So $h(w)$ has at least $\sum_{i=1}^s m_i+1=D+1$ preimages in $\Omega$ counted with multiplicities. This is a contradiction.
\endproof

\defi
Assume that $h$ is not constant.
We define 
the \emph{upper radius} of $(h,\Omega,x)$ to be
$$ \rho^\ast(h, \Omega, x):=\inf\left\{ r\geq 0: h(\Omega)\subseteq B(h(x),r)\right\}$$
and the \emph{proper lower radius} of $(h,\Omega,x)$ to be
$$\rho_\ast(h,\Omega, x):=\sup\left\{ r\geq 0: B(h(x),r)\subseteq_p h(\Omega,x) \right\}.$$
\enddefi

For convenience,  we define $\rho_\ast(h, \Omega, x)=\rho^\ast(h, \Omega, x):=0$ when $h$ is a constant map.
It is clear that $\rho^\ast(h, \Omega, x)\geq \rho_\ast(h,\Omega, x).$

\medskip

The above definition generalizes the usual notion of upper and lower radius for connected open subsets in $\P^1(\C).$
Let $U$ be a connected open subset of $\P^1(\C)$ and $a\in U.$
The \emph{upper radius} of $(U,a)$ is $$\rho^*(U,a):=\inf\left\{ r\geq 0: U\subseteq B(a,r)\right\}.$$
The \emph{lower radius} of $(U,a)$ is $$\rho_*(U,a):=\sup\left\{ r\geq 0:  B(a,r)\subseteq U\right\}.$$ 
Then $\rho^*(U,a)=\rho^*(\id, U,a) \text{ and } \rho_*(U,a)=\rho_*(\id, U,a).$
If $h$ is not constant, we have
$$\rho^*(h,\Omega, x)=\rho^*(h(\Omega),h(x)) \text{ and } \rho_*(h,\Omega, x)\leq \rho_*(h(\Omega),h(x)).$$
The equality holds if $h: \Omega\to h(\Omega)$ is proper.

\begin{pro}\label{proupdownrad} We have the following properties:
\begin{points}
\item Let $\Omega'\subseteq\Omega$ be a connected open neighborhood of  $x$. Then
$$\rho^\ast(h,\Omega', x)\leq \rho^\ast(h,\Omega, x) \text{ and } \rho_\ast(h,\Omega', x)\leq \rho_\ast(h,\Omega, x).$$
\item Let $\Omega_i, i\geq 0$ be an increasing sequence of connected open neighborhood of  $x$ satisfying $\cup_{i\geq 0}\Omega_i=\Omega$.
Then $$\rho^\ast(h,\Omega, x)=\sup_{i\geq 0}\rho^\ast(h,\Omega_i, x) \text{ and } \rho_\ast(h,\Omega, x)=\sup_{i\geq 0}\rho_\ast(h,\Omega_i, x).$$
\end{points}
\end{pro}
\proof
If $h$ is constant, the proposition is trivial. Now assume that $h$ is not constant.
Property (i) and the $\rho^*$ part of (ii) are obvious. We only prove the $\rho_*$ part of (ii). 
By (i), we have $\rho_\ast(h,\Omega, x)\geq \sup_{i\geq 0}\rho_\ast(h,\Omega_i, x).$
For every $r<\rho_\ast(h,\Omega, x)$, pick $r'\in (r, \rho_\ast(h,\Omega, x))$. There is an open neighborhood $W'$ of $x$ such that 
$h|_{W'}: W'\to B(h(x),r')$ is proper. Then $(h|_{W'})^{-1}\left(\overline{B(h(x),r)}\right)$ is compact. There is $i\geq 0$ such that $(h|_{W'})^{-1}\left(\overline{B(h(x),r)}\right)\subseteq \Omega_i$.  \par Set $W:=(h|_{W'})^{-1}(B(h(x),r))\subseteq \Omega_i.$
Since $h|_{W}: W\to B(h(x),r)$ is proper, $\rho_*(h,\Omega_i, x)\geq r,$ which concludes the proof.
\endproof

The following lemma shows that the upper and proper lower radii are stable under perturbations.     
\begin{lemma}\label{perturb}
Let $\Omega$ be a Riemann surface  and $x\in \Omega$.  
 For holomorphic maps
$g,h:\Omega\to \P^1(\C)$, define $\rho(h,g):=\sup_{z\in \Omega}\rho(h(z),g(z)).$ 
Then we have 
 \begin{equation}\label{equghup}
 \rho^\ast (h, \Omega, x)\leq \rho^*(g,\Omega,x)+2\rho(h,g).
\end{equation}
Assume further that $\rho^*(g,\Omega,x)+\rho(h,g)<\pi$, then we have 
\begin{equation}\label{equghdown}
\rho_\ast(h,\Omega, x)\geq \rho_\ast(g, \Omega, x)-2\rho(h,g).
\end{equation}
\end{lemma}
\begin{proof}
The first assertion is obvious. We now prove the second assertion. 
There is a sequence $\Omega_i,i\geq 0$ of open neighborhood of $x$ compactly contained in $\Omega$ and having smooth boundary such that 
$\cup_{i\geq 0} \Omega_i=\Omega$. By (ii) of Proposition \ref{proupdownrad}, 
we only need to prove (\ref{equghdown}) for each $\Omega_i.$
So we may assume that $\Omega$ is compactly contained in a connected Riemann surface $S$ with a smooth boundary and $h$ extends to a neighborhood of  $\overline{\Omega}.$
Set $\ep:=\rho(h,g).$

If $g$ is constant,  (\ref{equghdown}) is trivial. If $h$ is constant, by (\ref{equghup}), $\rho^*(g,\Omega,x)\leq 2\ep$. So $\rho_*(g,\Omega,x)\leq 2\ep$ which implies (\ref{equghdown}). Now assume that both $g$ and $h$ are not constant and $\rho_*(g,\Omega,x)>2\ep.$ 

Pick any $r\in (2\ep, \rho_*(g,\Omega,x))$, we claim that 
\begin{equation}\label{equbgxrpinh}B(g(x),r-\ep)\subseteq_p h(\Omega, x).
\end{equation}
Since $d(h(x), g(x))\leq \ep$, $B(h(x), r-2\ep)\subseteq B(g(x),r-\ep)$. Hence $r-2\ep\leq \rho_*(h,\Omega,x)$.
Let $r$ tend to $\rho_*(g,\Omega,x)$, then we get (\ref{equghdown}).

\medskip

We only need to prove the claim. 
We may assume that $g(x)=0.$ Since $\rho^*(g,\Omega,x)+\rho(h,g)<\pi$, $h(\Omega)\subseteq \P^1(\C)\setminus \{\infty\}.$
We identify $\P^1(\C)\setminus \{\infty\}$ with $\C$ and view $g,h$ as holomorphic functions on $\Omega.$
There is a connected neighborhood $W$ of $x$ such that $g|_W: W\to B(0,r)$ is proper. Then $W$ is compactly contained in $\Omega$ and has a piecewisely smooth boundary.
There is $D\geq 1$ such that for every $y\in B(0,r)$, it has exactly $D$ preimages under $g|_W$ counted with multiplicity.
\par For every $t\in [0,1]$, define $h_t:=th+(1-t)g.$ For every $z\in \Omega$, $h(z)\in \overline{B(g(z), \ep)}$. Note that $\overline{B(g(z), \ep)}$ is a disk in $\C$, though its center may not be $g(z).$ It follows that $h_t(z)\in \overline{B(g(z), \ep)}$ for every $t\in [0,1].$
Hence $d(g,h_t)\leq \ep$ for every $t\in [0,1].$
Since $g(\partial W)\subseteq \partial B(0,r)$, $$h_t(\partial W)\subseteq  \C\setminus B(0,r-\ep)$$
for every $t\in [0,1].$
In other words, for every $y\in B(0,r-\ep)$ and $t\in [0,1]$ there is no zero of $h_t-y$ in $\partial W$.
By argument principle, for every $y\in B(0,r-\ep)$, the number of preimages of $y$ under $h_t|_W$ is constant in $t$, hence equal to $D.$
Since $h_1=h$, for every $y\in B(0,r-\ep)$, $y$ has exactly $D$ preimages under $h|_W.$
By Lemma \ref{lempropercriterion},  $B(0,r-\ep)\subseteq_p h(\Omega,x),$ which concludes the proof.
\end{proof}

%

\subsection{Euclidean coordinates}
It is often easier to do the computation using Euclidean metric rather than the spherical metric. For this reason, we
 introduce an Euclidean coordinate $Z_a$ at each point $a\in \P^1(\C)$. 
Let $z$ be the standard coordinate on $\A^1=\Spec \C[z]\subseteq \P^1.$ We note that the spherical metric is invariant under the action of 
${\rm PU}(2,\C)<\PGL(2,\C).$
For every $a\in \P^1(\C)$, pick an element $H_a\in {\rm PU}(2,\C)$ such that $H_a(a)=0$. Define $Z_a:=2H_a^*z.$
Note that the choice of $H_a$ is unique up to composing a rotation $z\to \gamma z, |\gamma|=1$ by left.
Hence the induced coordinate $Z_a$ is unique up to multiplying some $\gamma$ with $|\gamma|=1.$
Note that the point defined by $Z_a=\infty$ is $H_a^{-1}(\infty)=a^{-}$.
Via $Z_a$, we identify $\P^1(\C)\setminus a^{-}$ with the standard complex plane $\C$ with origin $a.$
The spherical metric at $a$ is given by 
\begin{equation}\label{eqsmza}ds^2=\frac{1}{(1+1/4Z_a\overline{Z_a})^2}dZ_a\overline{dZ_a}
\end{equation}
We let $B(a,r)$ be the ball centered at $a$ of radius $r$ with respect to the spherical metric and set $\D(a,r):=\{Z_a<r\}.$
Then there is a strict increasing function $\tau:\R_{>0}\to\R_{>0}$ such that $B(0,\tau(r))=\D(0,r).$
Since $H_a$ preserves the spherical metric, for every $a\in \P^1(\C)$, $B(a,\tau(r))=\D(a,r).$
By (\ref{eqsmza}),  we have
\begin{equation}\label{equtau}\tau(r)=r+O(r^2)
\end{equation}
when $r\to 0.$

\bigskip

Let $g:\P^1(\C)\to \P^1(\C)$ be a rational map.

\begin{lemma}\label{distor1}
 There are $r_0>0$  and $C>0$ such that the following holds:  if a ball $B:=B(x,r)$ satisfies $r<r_0$, then we have 
\begin{equation}\label{equuprho}\rho^\ast(g, B, x)\leq |dg(x)|r+C r^2,
\end{equation}
and 
\begin{equation}\label{equdownrho}\rho_\ast(g, B, x)\geq |dg(x)|r-C r^2.
\end{equation}
\end{lemma}

\medskip

Note that (\ref{equdownrho}) is trivial if $ |dg(x)|r-C r^2\leq 0.$
\begin{proof}
By (\ref{equtau}), we only need to prove (\ref{equuprho}) and (\ref{equdownrho}) for Euclidean metric induced by the local coordinate $Z_a, a\in \P^1(\C).$

For every $a\in \P^1(\C)$, define $r_a:= \sup\{r\geq 0|\,\, g(\D(a,r))\subseteq \D(g(a),1)\}.$
Since $g$ is continuous, $r_a>0$ and the maps $a\in \P^1(\C)\mapsto r_a\in \R_{>0}$ is continuous.
Since $\P^1(\C)$ is compact, $c_0:=\min\{r_a, a\in \P^1(\C)\}>0$ exists.
For every $a\in \P^1(\C)$, in the local coordinates $Z_a, Z_{g(a)}$, $g$ takes form 
\begin{equation}\label{equatalor}g(Z)= \sum_{i\geq 1}A_i(a)Z^i.
\end{equation}
By (\ref{eqsmza}), $|A_1(a)|=|dg(a)|$.
By Cauchy integration formula, we have $|A_i(a)|\leq c_0^{-i}$ for every $a\in \P^1(\C)$ and $i\geq 1.$
Set $c_1:=c_0/10.$ For every $Z\in \D(a, c_1)$, 
we have 
$$|g(Z)|\leq |dg(a)||Z|+\sum_{i\geq 2}c_0^{-i}|Z|^i\leq |dg(a)||Z|+\frac{c_0^{-2}}{9}|Z|^2.$$
So for every $r<c_1$, 
$$g(\D(a,r))\subseteq \D\left(g(a),  |dg(a)|r+\frac{c_0^{-2}}{9}r^2\right).$$
It implies (\ref{equuprho}) via (\ref{equtau}).

Set $D_1:=c_0^{-2}/9$ and $D_2:=2D_1.$ By (\ref{equtau}), we only need to show that for every $r\in (0,c_1),$ if $|dg(a)|r-D_2r^2>0$, then
\begin{equation}\label{equinnerradius}
\D(g(a), |dg(a)|r-D_2r^2)\subseteq_p g(\D(a,r),a).
\end{equation}
We write $g$ as in (\ref{equatalor}). For $|Z|=r$ and $|Y|\leq |dg(a)|r-D_2r^2$, we have 
$$|(g(Z)-Y)-(A_1(a)Z-Y))|\leq \sum_{i\geq 2}c_0^{-i}r^i\leq D_1r^2$$ and
$$|A_1(a)Z-Y|\geq |dg(a)|r-(|dg(a)|r-D_2r^2)=D_2r^2>D_1r^2.$$
Since $A_1(a)Z-Y$ has exactly one zero in $\{|Z|<r\}$,
by Rouch\'e's theorem $g(Z)-Y$ has exactly zero in  $\{|Z|<r\}$.
By Lemma \ref{lempropercriterion}, we get (\ref{equinnerradius}), which concludes the proof.
\end{proof}

\subsection{Critical points}\label{criticalpoint}
Let $\mathcal{C}$ be the critical set of $g$.  For every $c\in \sC$, we let $l_c$ be its multiplicity. Then $\sum_{c\in \sC}l_c=2d-2.$ Set $l:=\max\{l_c|\,\, c\in \sC\}.$ 

\begin{lem}\label{lemdistorcritical}
There exists $r_1>0$  and $C_1>1$ such that the following holds:  if $c\in \sC$ and $r<r_1$
then we have 
\begin{equation}\label{equuprhoc}\rho^\ast(g,B(c,r),c)\leq C_1r^{l_c},
\end{equation}
and 
\begin{equation}\label{equdownrhoc}\rho_\ast(g, B(c,r),c)\geq C_1^{-1} r^{l_c}.
\end{equation}
Moreover, for every $a\in B(c,r)$, if $g(a)\in B(g(c), C_1^{-1} r^{l_c})$, then 
\begin{equation}\label{equctoaproper}B(g(c), C_1^{-1} r^{l_c})\subseteq_p g(B(c,r),a).
\end{equation}
\end{lem}
\proof 
By (\ref{equtau}), to show (\ref{equuprhoc}) and (\ref{equdownrhoc}), we only need to prove (\ref{equuprhoc}) and (\ref{equdownrhoc}) for Euclidean metric for the local coordinate $Z_a, a\in \P^1(\C).$

Let $c_0, c_1$ as in the proof of Lemma \ref{distor1}. As in the proof of Lemma \ref{distor1}, for every $c\in \sC$, in the local coordinates $Z_c, Z_{g(c)}$, $g$ takes form 
\begin{equation}\label{equftalor}g(Z)= \sum_{i\geq l_c}A_i(c)Z^i,
\end{equation}
where $A_{l_c}(c)\neq 0$ and $|A_i(c)|\leq c_0^{-i}$.

For every $Z\in \D(c, c_1)$, 
we have 
$$|g(Z)|\leq A_{l_c}(c)|Z|^{l_c}+\sum_{i\geq l_c+1}c_0^{-i}|Z|^i\leq \left(A_{l_c}(c)+c_0^{-l_c}/9\right)Z^{l_c}.$$
So for every $r<c_1$, 
$$g(\D(c,r))\subseteq \D(g(c),  (A_{l_c}(c)+c_0^{-l_c}/9)Z^{l_c}).$$
Since $\sC$ is finite, the above implies (\ref{equuprho}) via (\ref{equtau}).

\medskip
Set $A:=\min_{c\in \sC}|A_{l_c}(c)|.$ There is $c'>0$ such that for every $x\in \P^1(\C)$, $x$ has at most $l_c$ preimages counted with multiplicities in $\D(c,c').$
Pick $c_2:=\min\{c_1, \min\{1, c_0\}^{l}A/10, c'\}.$
Set $D_1:=c_0^{-l_c}\frac{c_2/c_1}{1-c_2/c_1}$ and  $D_2:=2D_1.$
We may check that 
\begin{equation}\label{equaid2c}|A_{l_c}(c)|-D_2>A/2.
\end{equation}
Since $\sC$ is finite, by (\ref{equtau}) and (\ref{equaid2c}), to prove (\ref{equinnerradiusc}) we only need to show that for every $r\in (0,c_2),$ 
\begin{equation}\label{equinnerradiusc}
\D(g(c), (|A_{l_c}(c)|-D_2)r^{l_c})\subseteq_p g(\D(c,r),c).
\end{equation}
We write $g$ as in (\ref{equftalor}). For $|Z|=r$ and $|Y|\leq (|A_{l_c}(c)|-D_2)r^{l_c}$, we have 
$$|(g(Z)-Y)-(A_{l_c}(c)Z^{l_c}-Y))|\leq \sum_{i\geq l_c+1}c_0^{-i}r^i\leq D_1r^{l_c}$$ and
$$|A_{l_c}(c)Z^{l_c}-Y|\geq |A_{l_c}(c)|r^{l_c}-(|A_{l_c}(c)|-D_2)r^{l_c}=D_2r^{l_c}>D_1r^{l_c}.$$
Since $A_{l_c}(c)Z^{l_c}-Y$ has exactly $l_c$ zeros in $\{|Z|<r\}$,
by Rouch\'e's theorem $g(Z)-Y$ has exactly $l_c$ zeros in  $\{|Z|<r\}$. By Lemma \ref{lempropercriterion}, we get 
(\ref{equinnerradiusc}).

We now prove (\ref{equctoaproper}).
Let $a\in B(c,r)$ with $g(a)\in B(g(c), C_1^{-1} r^{l_c})$.
By (\ref{equinnerradiusc}), there is an open neighborhood $W$ of $c$ in $B(c,r)$ such that $g|_W: W\to B(g(c), C_1^{-1} r^{l_c})$ is proper.
Since $g(c)$ has exactly one preimage $c$ of multiplicity $l_c$, $g(a)$ has $l_c$ preimages with multiplicity in $W.$ Since $r\leq c'$, $g(a)$ has at most $l_c$ preimages with multiplicity in $B(c,r).$ Hence $g^{-1}(g(a))\cap B(c,r)=g^{-1}(g(a))\cap W.$ Then $a\in W$, which implies (\ref{equctoaproper}).
\endproof

Set 
\begin{equation}\label{equdefinedelta}\delta:=\min\{d(c_1,c_2)|\,\, c_1,c_2 \text{ are distint points in }\sC\}/3.
\end{equation}
Then for every $a\in \P^1(\C)$ and $r\in (0,\delta]$ there is at most one $c\in \sC\cap B(a,r).$
For every $a\in \P^1(\C)$, define $l_a$ as follows: if $d(a,\sC)\geq \delta$, set $l_a:=0$; otherwise, let $c_a$ be the unique critical point with $d(a,c_a)<\delta$ and $l_a:=l_c.$
This extends our previous definition of $l_c$ for $c\in \sC.$
Note that, for every $a\in \P^1(\C)$ either for every $b\in B(a,\delta/100),$ $l_b=l_a$ or 
$d(B(a,\delta/100),\sC)>\delta/2.$
So there is a constant $A_1>1$ such that for every $a\in \P^1(\C)$ and $b\in B(a,\delta/100),$ we have 
\begin{equation}\label{equdgcriticaldis}
A_1^{-1} d(x,\mathcal{C})^{l_b-1}\geq |dg(a)|\geq A_1 d(x,\mathcal{C})^{l_b-1}.
\end{equation}

\medskip
\par The following lemma is a corollary of Lemma \ref{distor1}.
\begin{lemma}\label{distor2}
There exists $r_2>0$ and $C_2>0$ such that the following holds:  if two balls $B:=B(x,r)$, $B':=B(x,r')$ satisfy $r'<r<r_2$, and $d(x,\mathcal{C})^l>r$, then we have 
\begin{equation}\label{4.3}
 \frac{\rho_\ast(g, B', x)}{\rho^\ast(g, B,x)}\geq  \frac{r'}{r}-C_2r^{1/l}
\end{equation}
and 
\begin{equation}\label{4.4}
	\frac{\rho^\ast(g, B',x)}{\rho_\ast(g,B, x)}\leq   \frac{r'}{r}+C_2r^{1/l}.
\end{equation}
\end{lemma}
\begin{proof}
Let $r_0$, $C$ as in Lemma \ref{distor1}.
We may further ask that $r_0<\min\{1, A_1/C\}.$ Then under the assumption $r<r_0$ and  $d(x,\mathcal{C})^l>r$, for every $u\in (0,r]$, 
we have $|dg(x)|u-Cu^2>0.$

\medskip

By Lemma \ref{distor1} we have 
\begin{align*}
 \frac{\rho_\ast(g,B',x)}{\rho^\ast(g,B,x)}&\geq \frac{|dg(x)|r'-Cr'^2}{|dg(x)|r+Cr^2}
 \\&=\frac{r'}{r}-\frac{Cr'(r-r')}{r|dg(x)|+Cr^2}
 \\&\geq \frac{r'}{r}-\frac{Cr^2/4}{A_1r^{2-1/l}+Cr^2}
\\&\geq \frac{r'}{r}- C_2 r^{1/l},
\end{align*}
where $C_2>0$ is a constant. This implies (\ref{4.3}).  Similarly one can prove (\ref{4.4}). 
\end{proof}

\begin{lem}\label{leminjectiveball}There is a constant $A_2>1$ such that for every point $x\in \P^1(\C)\setminus \sC$, 
$g|_{B(x, d(x,\mathcal{C})/A_2)}$ is injective.
\end{lem}
\proof For every $c\in \sC$, there is an open neighborhood $U_c$ of $c$ such that 
such that there are isomorphisms $\phi_c: U_c\to \D$ and $\psi_c: g(U_c)\to \D$, such that $\phi_c(c)=0,\psi(f(c))=0$ and $G_c:=\psi_c\circ g\circ (\phi_c)^{-1}: z\to z^{l_c}.$ Recall that $l_c$ is the multiplicity of $c$.
After shrinking $U_c$, we may assume that $U_c\subseteq B(c,\delta)$ and on $U_c$ and  $g(U_c)$, the spherical metrics are equivalent to the metrics induced by Euclidean
metric on $\D$ via $\phi_c$ and $\psi_c$ respectively.  
There is $D>1$ such that for every $c\in \sC$ and $x,y\in U_c, z,w\in g(U_c)$, we have
\begin{equation}\label{equequidis}D^{-1}|\phi_c(x)-\phi_c(y)|\leq d(x,y)\leq D|\phi_c(x)-\phi_c(y)|
\end{equation}
Set $V_c:=\phi^{-1}_c(\D(0,1/2)).$
and  $K:=\P^1(\C)\setminus (\cup_{c\in \sC}V_c).$
For every $x\in K$, there is $r_x>0$ such that $g|_{B(x,r_x)}$ is injective.  Since $K$ is compact, there is a finite subset $F\subseteq K$ such that $K\subseteq \cup_{x\in F}B(x,r_x).$
There is $\delta_1>0$ such that for every $x\in K$, there is $y\in F$ such that 
\begin{equation}\label{equink}B(x,\delta_1)\subseteq B(y, r_y).
\end{equation}

\medskip

For every $x\in V_c$, $|\phi_c(x)|<1/2$, it is clear that $G_c|_{\D(\phi_c(x),|\phi_c(x)|/(100l))}$ is injective. 
By (\ref{equequidis}), we have 
$$B(x,d(x,\sC)/(100D^2l))\subseteq B(x, |\phi_c(x)|/(100Dl))$$
and $$\phi_c(B(x, |\phi_c(x)|/(100Dl)))\subseteq \D(\phi_c(x),|\phi_c(x)|/(100l)).$$
So $g|_{B(x,d(x,\sC)/(100D^2l))}$ is injective.  Set $A_2:=\max\{100D^2l, 2\pi/\delta_1\},$ we conclude the proof by (\ref{equink}).
\endproof

Set $A_3:=2A_2,$ where $A_2$ is the constant in Lemma \ref{leminjectiveball}. By Koebe distortion theorem and (\ref{equdgcriticaldis}), there is $C_3>1$ such that for every $x\in \P^1(\C)\setminus \sC$ and $r\leq d(x,\sC)/A_3$
we have 
\begin{equation}\label{equupouc}
\rho^*(g, B(x,r),x)\leq C_3d(x,\sC)^{l_x-1}r 
\end{equation}
and 
\begin{equation}\label{equdownouc}
 \rho_*(g, B(x,r), x)\geq C_3^{-1}d(x,\sC)^{l_x-1}r.
\end{equation}

\par Without assuming $d(x,\mathcal{C})^l>r$, we also have the following weaker distortion estimates.
\begin{lemma}\label{distor3}
There exists $r_3>0$ and $\theta>1$ such that the following holds:  if two balls $B:=B(x,r)$, $B':=B(x,r')$ satisfy $r'<r<r_3$, then we have 
\begin{equation}\label{4.5}
	\frac{\rho_\ast(g, B', x)}{\rho^\ast(g, B, x)}\geq  \frac{1}{\theta} \frac{(r')^l}{r^l}
\end{equation}
and 
\begin{equation}\label{4.6}
	\frac{\rho^\ast(g, B',x)}{\rho_\ast(g,B,x)}\leq \theta \frac{r'}{r}.
\end{equation}
\end{lemma}
\begin{proof}
The proof is based on the following lemma.
\begin{lem}\label{lembounupdownc}
There are $p_0>0$ and  two constants $\theta_1>0$ and $\theta_2>0$ such that for $r<p_0$, the following holds:
\begin{points}
\item If a ball $B:=B(x,r)$ satisfies $r<d(x,\mathcal{C})/2$, then
\begin{equation}\label{eqursmup}\rho^\ast(g,B,x)\leq \theta_1 d(x,\mathcal{C})^{l_x-1}r
\end{equation} and 
\begin{equation}\label{eqursmdown}\rho_\ast(g, B,x)\geq \theta_2 d(x,\mathcal{C})^{l_x-1}r.
\end{equation}
\item If a ball $B:=B(x,r)$ satisfies $r\geq d(x,\mathcal{C})/2$, then
\begin{equation}\label{equrlaup}\rho^\ast(g, B, x)\leq \theta_1 r^{l_x}
\end{equation} and 
\begin{equation}\label{equrladown}\rho_\ast(g, B,x)\geq \theta_2 r^{l_x}.
\end{equation}
\end{points}
\end{lem}

\par Set $\theta:=2\theta_1/\theta_2$. To show (\ref{4.5}) and (\ref{4.6}), there are three cases. 
\medskip
\par {\em Case 1}: we have $r<d(x,\mathcal{C})/2$. Then 
\begin{equation*}
	\frac{\rho_\ast(g,B',x)}{\rho^\ast(g,B,x)}\geq  \frac{\theta_2 d(x,\mathcal{C})^{l_x-1}r'}{\theta_1 d(x,\mathcal{C})^{l_x-1}r}\geq \frac{1}{\theta}\frac{r'}{r},
\end{equation*}
and 
\begin{equation*}
	\frac{\rho^\ast(g,B',x)}{\rho_\ast(g,B,x)}\leq    \frac{\theta_1 d(x,\mathcal{C})^{l_x-1}r'}{\theta_2 d(x,\mathcal{C})^{l_x-1}r}\leq \theta \frac{r'}{r}.
\end{equation*}
\medskip
\par {\em Case 2}: we have $r'<d(x,\mathcal{C})/2$ but  $r\geq d(x,\mathcal{C})/2$. Then
\begin{equation*}
	\frac{\rho_\ast(g,B',x)}{\rho^\ast(g,B,x)}\geq  \frac{\theta_2 d(x,\mathcal{C})^{l_x-1}r'}{\theta_1 d(x,\mathcal{C})^{l_x}}\geq \frac{1}{\theta} \frac{r'}{r},
\end{equation*}
and 
\begin{equation*}
	\frac{\rho^\ast(g,B',x)}{\rho_\ast(g,B,x)}\leq    \frac{\theta_1 d(x,\mathcal{C})^{l_x-1}r'}{\theta_2 d(x,\mathcal{C})^{l_x}}\leq \theta\frac{r'}{r}.
\end{equation*}
\medskip
\par {\em Case 3}: we have $r'\geq d(x,\mathcal{C})/2$. Then
\begin{equation*}
	\frac{\rho_\ast(g,B',x)}{\rho^\ast(g,B,x)}\geq  \frac{\theta_2 (r')^{l_x}}{\theta_1 r^{l_x}}\geq \frac{1}{\theta} \frac{(r')^l_x}{r^l_x},
\end{equation*}
and 
\begin{equation*}
	\frac{\rho^\ast(g,B',x)}{\rho_\ast(g,B,x)}\leq    \frac{\theta_1 (r')^{l_x}}{\theta_2 r^{l_x}}\leq \theta\frac{r'}{r} .
\end{equation*}
\end{proof}

\proof[Proof of Lemma \ref{lembounupdownc}]
By (\ref{equdgcriticaldis}), for $r<\delta/100$ and every $y\in B(x,r),$ we have $$|dg(y)|\leq A_1 d(y,\mathcal{C})^{l_x-1}.$$
If $r<d(x,\sC)/2$, then $d(y,\mathcal{C})\leq 3/2d(x,\mathcal{C}).$
So $$g(B(x,r))\subseteq B(g(x),A_1 (3/2)^{l-1}d(x,\mathcal{C})^{l_x-1}r),$$ which implies (\ref{eqursmup}).

Next, we prove (\ref{eqursmdown}). Since $r<d(x,\mathcal{C})/2$, $r/A_3<d(x,\mathcal{C})/A_3.$ By 
(\ref{equdownouc}), we have 
$$\rho_*(g, B(x,r), x))\geq \rho_*(g, B(x,r/A_3), x))\geq (C_3A_3)^{-1}d(x,\mathcal{C})^{l_x-1}r.$$
This implies (\ref{eqursmdown}).

Now we assume that $r\geq d(x,\mathcal{C})/2$.  
Let $r_1, C_1$ as in Lemma \ref{lemdistorcritical}. Assume that $r<\min\{r_1,\delta, 1/C_1\}/100.$ Then $d(x,\mathcal{C})<\min\{r_1,\delta,1/C_1\}/50$. Set $c:=c_x.$ We have $l_x=l_c.$ Recall that $l_c$ is the multiplicity of $c$.
Since $B(x,r)\subseteq B(c, 3r)$, $g(B(x,r))\subseteq g(B(c,3r)).$ 
By Lemma \ref{lemdistorcritical}, $$g(B(x,r))\subseteq g(B(c,3r))\subseteq B(g(c), C_13^l_cr^{l_c}).$$ 
Since $d(c,x)\leq 2r$, by Lemma \ref{lemdistorcritical}, we have $d(g(c),g(x))\leq C_12^l_cr^{l_c}.$
Then we have $$g(B(x,r))\subseteq B(g(c), C_13^l_cr^{l_c})\subseteq B(x, C_1(3^{l_c}+2^{l_c})r^{l_c})).$$
Since $C_1(3^{l_c}+2^{l_c})r^{l_c})\leq C_1(3^{l}+2^{l})r^{l_c})$, we get (\ref{equrlaup}).

Finally, we prove (\ref{equrladown}). We first treat the case where $r\leq 10(C_1^2+1)d(x,\mathcal{C}).$
Since $r\geq d(x,\mathcal{C})/2$ and $A_3>2$, 
$B(x,r)$ contains $B(x, d(x,\mathcal{C})/2).$
By (\ref{equdownouc}), 
\begin{align*}\rho_*(g(B(x,r)),g(x))\geq &\rho_*(g(B(x,d(x,\mathcal{C})/A_3)), g(x))\\ 
\geq &(A_3C_3)^{-1}(d(x,\mathcal{C}))^{l_x}\\
\geq & (A_3C_3)^{-1}(r/(10(C_1^2+1)))^{l_x},
\end{align*}
which implies (\ref{equrladown}).
Now assume that $r>(10(C_1^2+1))d(x,\mathcal{C})$.
Note that $d(x,\sC)=d(x,c)$. 
Set $Q:=(10(C_1^2+1)).$ Since $r>Qd(x,\mathcal{C})$,
\begin{equation}\label{equdchangr}
C_1^{-1}(r-d(x,c))^{l_x}\geq C_1^{-1}\left(\frac{Q-1}{Q}\right)^{l_x}r^{l_x}.
\end{equation}
By Lemma \ref{lemdistorcritical},
we have $$d(g(c),g(x))\leq C_1d(c,x)^{l_x}\leq C_1Q^{-l_x}r^{l_x}.$$
One may check that $C_1Q^{-l_x}\leq 1/10 C_1^{-1}(\frac{Q-1}{Q})^{l_x}.$
We get 
\begin{equation}\label{equbxbc}B(g(x),9C_1Q^{-l_x}x^{l_x})\subseteq B(g(c), C_1^{-1}(r-d(x,c))^{l_x})
\end{equation}
and 
$$g(x)\in B(g(c), C_1^{-1}(r-d(x,c))^{l_x}).$$
Since $r>(10(C_1^2+1))d(x,c)>10d(x,c)$, $x\in B(c,r-d(x,c)).$
Then by (\ref{equctoaproper}) of Lemma \ref{lemdistorcritical}, we get 
$$B(g(c), C_1^{-1}(r-d(x,c))^{l_x})\subseteq_p g(B(c,r-d(x,c)),x).$$
Since   $B(c,r-d(x,c))\subseteq B(x,r),$
$$B(g(c), C_1^{-1}(r-d(x,c))^{l_x})\subseteq_p g(B(x,r),x).$$
By (\ref{equbxbc}), $B(g(x),9C_1Q^{-l_x}x^{l_x})\subseteq_p g(B(x,r),x)$, which concludes the proof.
\endproof

\section{Bounded distortion for non-uniformly hyperbolic maps}\label{plough 2}
In this section, we show some nice bounded distortion properties of Topological Collet-Eckmann and Polynomial Recurrence  maps.
\medskip
\par Let $q:=\#(\mathcal{C}\cap \sJ(g))$. 
Define $d_1(\cdot,\cdot):=\min\{d(\cdot,\cdot),1\}.$  We have $d_1(x,y)\leq d(x,y)$ for every $x,y\in\P^1(\C)$.
The following lemma is \cite[(3.3) in the proof of Lemma 3.4 ]{denker1996transfer}.
\begin{lemma}[Denker-Przytycki-Urbanski \cite{denker1996transfer}]\label{dpu}
There exists $Q>0$ such that for every $x\in \sJ(g)$ and $n\geq 1$, the following holds:
\begin{equation*}
\sum_{\substack{0\leq k\leq n-1, \\\text{except}\;q\;\text{terms}}} -\log d_1(g^k(x),\mathcal{C})\leq Qn.
\end{equation*}
\end{lemma}
\medskip
\par   Let $\delta_0>0$. For every fixed $x\in \sJ(g)$ and $n\geq 1$, for $0\leq k\leq n$ we define $W_k(n)$ to be the connected component of $g^{k-n}(B(g^n(x),\delta_0))$ containing $g^k(x)$.  When $n$ is clear, we write $W_m$ for the simplicity. The following lemma is inspired by Przytycki-Rohde \cite{przytycki1998porosity}. 
\begin{lemma}\label{pullback}
Assume $g$ is $\TCE(\la)$ for some $\la>1$, $\delta_0>0$.  Let $x\in \sJ(g)$. Then for every $\ep>0$, there exists $N>0$ and a subset $A\subseteq \Z_{\geq 0}$ satisfying $\underline{d}(A)>1-\ep$ such the following holds: for every $m\in A$,  and $k\not\in E_m$ where $E_m$ is a subset of $\{0,\dots,m\}$ containing at most $N$ elements, we have
$$d_1(g^k(x),\mathcal{C})^l>\diam W_k(m). $$
\end{lemma}
\begin{proof}
For each $k\in \Z_{\geq 0}$, let $I_k$ be the closed interval $$I_k:=\left[k,k+\frac{l}{\log \la}(-\log d_1(g^k(x),\mathcal{C})) \right].$$
By Lemma \ref{dpu},  for every $n\geq q$ there is a subset $F_n\subseteq \{0,\dots, n-1\}$ with $\#F_n=q$ such that 
$$\sum_{0\leq k\leq n-1, k\not\in F_n} |I_k\cap [0,n]|\leq \frac{lQ}{\log \la} n.$$
For every $N\geq 1$, set 
$$A_N:=\left\{ n\in \Z_{\geq 0}: \text{there are at most}\; N\; \text{intervals}\;I_k\;\text{containing}\; n\right\}.$$
For every $i\in \{0,\dots,n-1\}\setminus A_N$, $i$ is covered by at least $N+1$ intervals $I_k$.
Since the left endpoint of those $I_k$ are distinct integers, at least $N$ left endpoints are $\leq k-1.$
Hence  there are at least $N-q$ intervals among $I_k, k\in \{0,\dots,n-1\}\setminus F_n$ covers the interval $[i-1,i]$.
Hence $$(N-q)(n-\#(A_N\cap [0,n-1]))\leq \frac{lQ}{\log \la} n.$$ It follows that 
$$\#(A_N\cap [0,n-1])\geq \left(1- \frac{lQ}{\log \la(N-q)} \right)n.$$
Pick $N$ large enough, we have $\underline{d}(A_N)>1-\ep$.  We set $A:=A_N$. We need to show that $A$ satisfies the property we want. It suffices to show for $m\in A$,  if $m\notin I_k$  for $0\leq k\leq m$, then we have $d(g^k(x),\mathcal{C})^l>\diam W_k. $ The condition $m\notin I_k$ and the the $\TCE(\la)$ property imply that 
$$ d_1(g^k(x),\mathcal{C})^l>\la^{k-m}\geq \diam W_k.$$
 This finishes the proof. 
\end{proof}
\medskip
\par If in addition a point $x\in\sJ(g)$ satisfies $\PR(s)$, we  have the following two lemmas.
\begin{lemma}\label{pr1}
Assume $g$ is $\TCE(\la)$ for some $\la>1$, $\delta_0>0$.  Let $x\in \sJ(g)$ satisfy $\PR(s)$ for some $s>0$ and $x$ is not a preimage of a critical point.  Then the following holds:  for every large $n\geq 1$,  if some $k\in \{0,\dots, n-1\}$ sartisfies $d(g^k(x),\mathcal{C})^l\leq \diam W_k(n)$, then $k\geq n-(sl/\log \la)\log n$. 
\end{lemma}
\begin{proof}
By $\PR(s)$, there is a constant $N_0>0$ such that  
\begin{equation}\label{equprcon}d(g^k(x),\mathcal{C})>k^{-s}
\end{equation} for every $k\geq N_0$.  
There is $N_1>0$ such that for every $k=0,\dots, N_0$, 
\begin{equation}\label{equsmnpr} d(g^k(x),\mathcal{C})^l>\la^{-(N_1-N_0)}.
\end{equation}
Let  $n\geq N_1,$ and $k\in \{0,\dots, n\}.$
Assume that $d(g^k(x),\mathcal{C})^l\leq \diam W_k.$ 
Since $\diam W_k\leq \la^{k-n}$, 
we have 
\begin{equation}\label{equasstce}d(g^k(x),\mathcal{C})^l\leq \la^{k-n}.
\end{equation}
By (\ref{equsmnpr}),  we get
$k\geq N_0+1.$
Then by (\ref{equprcon}), we get 
$$k\geq n-(sl/\log \la)\log n,$$
which concludes the proof.
\end{proof}
\medskip
\begin{lemma}\label{pr2}
Assume that $g$ is $\TCE(\la)$ for some $\la>1$. Let $\delta_0>0$ be small enough.  Let $x\in \sJ(g)$ satisfy $\PR(s)$ for some $s>0$ and $x$ is not a preimage of a critical point.  Let $\ep>0$ and let $A$ be the subset defined in Lemma \ref{pullback}. Then 
\begin{points}
\item There is a constant $C>0$ such that the following holds:  for every $m\in A$, we have
\begin{equation*}
	\sum_{k=0}^m \frac{\diam W_0(m)}{\diam W_k(m)}<C.
\end{equation*}
\item  For every $\eta>0$ there exists $N_0>0$  such that the following holds:  for every $m\in A$, $m\geq N_0$ we have
\begin{equation*}
	\sum_{k=N_0}^m \frac{\diam W_0(m)}{\diam W_k(m)}<\eta.
\end{equation*}
\end{points}
\end{lemma}
\begin{proof}
Since $x$ satisfies $\PR(s)$, $g$ is $\TCE(\la)$ and $x$ is not a preimage of a critical point, 
by \cite[Lemma A.4]{ji2023non},
there exists $\la_1>1$ and $C'>0$ such that for every $n\geq 1$,
$|dg^n(x)|\geq C'\la_1^n.$

\par 
Since $g$ is  $\TCE(\la)$, for $\delta_0$ small enough, we may assume that all $W_i(m)$ has diameter at most $\delta$ (see (\ref{equdefinedelta})).
So every $W_i(m)$ meets at most one critical point. Hence all $W_i(m)$ are simply connected.
For each $i=0,\dots,m$, the map $g^{m-i}|_{W_i}: W_i\to W_m$ is proper with at most $l^N$ critical points counted with multiplicity.
By the Koebe type distortion property for proper holomorphic maps with a bounded number of critical points, see \cite[Lemma 2.1]{przytycki1998porosity},  there exists a uniform constant $\beta>0$ such that 
\begin{equation}\label{equPR98}\rho_\ast(W_i(m),x)>\beta \diam W_i(m).
\end{equation}
\par By Lemma \ref{pr1}, for $k< m-(sl/\log \la)\log m$, $g^k:W_0\to W_k$ is injective.  By Koebe one-quarter theorem, we have 
\begin{equation}\label{4.12}
\diam W_k\geq |dg^k(x)|\rho_\ast(W_0,x)/4\geq C_4\la_1^k \diam W_0,
\end{equation}
where $C_4>0$ is a constant.
\par On the other hand there exists a constant $L>1$ such that for every $0\leq k\leq m-1$ we have 
\begin{equation}\label{4.13}
\diam W_k\geq L^{k-m} \delta_0.
\end{equation}
\par Let $p:=\lfloor m-(sl/\log \la)\log m\rfloor$. Combine with (\ref{4.12}) and (\ref{4.13}) there exists $C_5>0$ such that 
\begin{align*}
\sum_{k=0}^m \frac{\diam W_0}{\diam W_k}&=\sum_{k=0}^{p}  \frac{\diam W_0}{\diam W_k}+ \sum_{k=p+1}^{m}  \frac{\diam W_0}{\diam W_k}
\\&\leq \sum_{k=0}^{\infty} \frac{1}{C_4\la_1^k}+\sum_{k=p+1}^{m}\frac{\la^{-m}}{ L^{k-m} \delta_0}
\\&\leq \sum_{k=0}^{\infty} \frac{1}{C_4\la_1^k}+\frac{\la^{-m}}{\delta_0(L-1)}L^{m-p}
\\&\leq \sum_{k=0}^{\infty} \frac{1}{C_4\la_1^k}+\frac{\la^{-m}}{\delta_0(L-1)}L^{(sl/\log \la)\log m+1}
\\&<C_5.
\end{align*}
\par This proves (i). To show (ii), similarly we have 
\begin{align*}
	\sum_{k=N_0}^m \frac{\diam W_0}{\diam W_k}&=\sum_{k=N_0}^{p}  \frac{\diam W_0}{\diam W_k}+ \sum_{k=p+1}^{m}  \frac{\diam W_0}{\diam W_k}
	\\&\leq \sum_{k=N_0}^{\infty} \frac{1}{C_4\la_1^k}+\sum_{k=p+1}^{m}\frac{\la^{-m}}{ L^{k-m} \delta_0}
	\\&<\eta,
\end{align*}
for $N_0$ large enough. This finishes the proof.
\end{proof}

\medskip

\section{From phase space to parameter space}\label{harvest}
\proof[Proof of Theorem \ref{renor}]
By Lemma \ref{precri}, after replacing $a$ by a suitable iterate, we may assume  that the orbit of $a(0)$ does not intersect $\mathcal{C}_0$. 
The $\PCE(\la_0)$ condition shows that there exists $\gamma\neq 0$ such the transversality condition in Lemma \ref{trans} holds. 
After replacing $a$ by a suitable iterate, we may assume that 
\begin{equation}\label{equpcetens}|d\xi_{a,n}/dt(0)|/|d(g^n)(z)|\in (|\gamma|/2, 2|\gamma|).
\end{equation}
Set $C_6:=\sup_{t_0\in \D}|d\xi_{a,0}/dt(t_0)|.$
Set $g:=f_0$ and $z:=a(0)$. By Theorem \ref{DGV}, $0$ is marked Collet-Eckmann parameter, in particular $z\in \sJ(g)$. Moreover, $z$ is not a preimage of a critical point. For an arbitrary fixed $\ep>0$, let $A$ be the subset defined in Lemma \ref{pullback}. To show $A$ has the properties we want, it suffices to construct $\left\{\rho_m \right\}$ such that for every $m\in A$ large, the following  holds: 
\begin{equation}\label{4.7}
h_m(\D)\subseteq B(g^m(z),\delta_0),
\end{equation}
and there exists $\delta_1>0$ such that 
\begin{equation}\label{4.8}
B(g^m(z),\delta_1)\subseteq h_m\left(\frac{1}{2}\D\right).
\end{equation}

\medskip

Since $g$ is  $\TCE(\la)$, for $\delta_0$ small enough, we may assume that for every $m\geq 0, i=0,\dots,m$, 
\begin{equation}\label{equwidiam}\diam(W_i(m))\leq \min\{\delta, 1/10\} 
\end{equation}
where $\delta$ is defined in (\ref{equdefinedelta}). Hence all $W_i(m)$ are simply connected.

\medskip

For $0\leq k\leq m$,  set $R_k:=\rho_\ast (W_k, g^k(z))$
 and $D:=B(z,R_0)$.  For $0\leq k\leq m$, set  $R'_k:=\rho^\ast (g^k(D),g^k(z))$. 
We first show that $R_k,R_k'$, and $\diam W_k$ are comparable.  By (\ref{equPR98}), we have $\beta \diam W_k \leq R_k\leq \diam W_k.$
By \cite[(2.2) of Lemma 2.1]{przytycki1998porosity},  there is  $C_{\tau}>0, \tau\in (0,1/2)$ with $C_{\tau}\to 0$ as $\tau\to 0$ such that the following holds:
For every $r\leq \tau R_k$, let $W''(\tau)$ be the connected component of $g^{-k}(B(g^k(z),r))$ containing $z$, then we have $\diam W''(\tau)<C_{\tau} \diam W_0.$
Pick $\beta_1\in (0,1/2)$ such that $C_{\beta_1}<\beta.$ Then 
$\diam W''(\beta_1)<\beta \diam W_0\leq R_0.$
Hence $W''(\beta_1)\subseteq D$.
Note that $g^k(D)\subseteq W_k$,  we get
\begin{equation}\label{equRk'bound}\beta\beta_1\diam W_k\leq \beta_1R_k\leq \rho_*(g^k(D),g^k(z))\leq R_k'\leq \diam W_k
\end{equation}

\medskip

Let $L:=2\sup_{x\in \frac{1}{2}\D\times \P^1(\C)} |df(x)|$. 
For every $t\in \frac{1}{2}\D$, we have 
\begin{equation}\label{equnearbyfiberd}
d(\xi_{a_{k+1}}(t),g(\xi_{a,k}(t)))=d(f_t(\xi_{a,k}(t)),g(\xi_{a,k}(t)))\leq L|t|
\end{equation}

Let $E_m\subseteq \left\{ 0,1,\dots, m-1\right\}$ be the exceptional set  as in Lemma \ref{pullback}.
We have $\#E_m\leq N$.   Set $r_k:=\rho^\ast  (\xi_{a,k}(\rho_m \D),g^k(z)) $ and $B_k:=B(g^k(z),r_k)$. 
By (\ref{equnearbyfiberd}), if $\rho_m<1/2$, we have 
\begin{equation}\label{4.9}
r_{k+1}\leq \rho^{\ast}(g, B_k, g^{k}(z))+L\rho_m,
\end{equation}
Hence 
\begin{equation}\label{4.9'}
\begin{split}
\frac{r_{k+1}}{R_{k+1}}\leq& \frac{\rho^*(g, B_k, g^k(z))}{\rho_*(W_{k+1}, g^{k+1}(z))}+\frac{L\rho_m}{R_{k+1}}\\
\leq& \frac{\rho^*(g, B_k, g^k(z))}{\rho_*(g, B(g^k(z),R_k), g^{k}(z))}+\frac{L\rho_m}{R_{k+1}}.
\end{split}
\end{equation}

For $0\leq k\leq m-1$ such that $k\notin E_m$, if 
\begin{equation}\label{equcondRkrk}
\frac{r_k}{R_k}<1,
\end{equation}
then by (\ref{4.4}) of Lemma \ref{distor2} and (\ref{4.9'}) we have
\begin{equation}\label{4.10}
\frac{r_{k+1}}{R_{k+1}}\leq\frac{r_k}{R_k}+C_2R_k^{1/l}+\frac{L\rho_m}{R_{k+1}}
\end{equation}
If  $k\in E_m$, and (\ref{equcondRkrk}) holds,
by (\ref{4.6}) of Lemma \ref{distor3}  and (\ref{4.9'}) we have
\begin{equation}\label{4.11}
	\frac{r_{k+1}}{R_{k+1}}\leq\theta  \frac{r_k}{R_k}+\frac{L\rho_m}{R_{k+1}}.
\end{equation}
If (\ref{equcondRkrk}) holds for every $i=0,\dots,k-1$,
by (\ref{4.10}) and (\ref{4.11}) we have for every $0\leq i\leq m$
\begin{equation}\label{equrioverRi}
\begin{split}
	\frac{r_{i}}{R_{i}}&\leq\theta^N \sum_{k=0}^{i-1}\left( C_2R_k^{1/l}+\frac{L\rho_m}{R_{k+1}}\right)+\theta^N \frac{r_0}{R_0}
	\\&\leq\theta^N \sum_{k=0}^{m-1}\left( C_2R_k^{1/l}+\frac{L\rho_m}{R_{k+1}}\right)+\theta^N \frac{r_0}{R_0}
	\\&\leq \theta^N \sum_{k=0}^{m-1}\left( C_2R_k^{1/l}+\frac{L\rho_m}{R_{k+1}}\right)+\theta^N \frac{C_6\rho_m}{R_0}.
\end{split}
\end{equation}
\par Since $R_k\leq \la^{k-m}$, shrink $\delta_0$ if necessary we may assume that 
\begin{equation}\label{equthncrk}\theta^N \sum_{k=0}^{m-1} C_2R_k^{1/l}<\beta\beta_1/4. 
\end{equation}
By Lemma \ref{pr2} (i), there exists a constant $\alpha \in (0, 1/(10L+10\pi))$ such that
\begin{equation}\label{equdefial}\theta^N \frac{C_6\alpha \diam W_0(m)}{R_0}+\theta^N \sum_{k=0}^{m-1} \frac{L\alpha \diam W_0(m)}{R_{k+1}}<\beta\beta_1/4. 
\end{equation}
We define $\rho_m:=\alpha \diam W_0(m)$, then $\rho_m<1/2.$ 
Hence (\ref{4.9}) holds. Moreover, by the $\TCE(\la)$ condition,
$\rho_m\to 0$ as $m\to \infty.$
Since $r_0\leq C_6\times \rho_m=C_6\alpha\diam W_0$, by (\ref{equdefial})
$r_0/R_0<\beta\beta_1/4<1.$
Apply (\ref{equcondRkrk}), (\ref{equrioverRi}), (\ref{equthncrk}) and (\ref{equdefial}) inductively, we get 
\begin{equation}\label{equrirri}r_{i}/R_{i}<\beta\beta_1/2
\end{equation} 
for every $i=0,\dots, m$, which implies (\ref{4.7}).  By (\ref{equwidiam}), for every $m\geq 0$,
\begin{equation}\label{condiepplrhost}\rho_mL\leq 1/10.
\end{equation}
\par It remains to show for our choice $\rho_m:=\alpha \diam W_0(m)$, there exists $\delta_1>0$ such that (\ref{4.8}) holds.  For $0\leq k\leq m$ we set $r'_k:=\rho_\ast (\xi_{a,k},(\rho_m/2)\D, 0)$ and $B'_k:=B(g^k(z),r'_k)$. 
By (\ref{equrirri}) and (\ref{equRk'bound}),
we get 
\begin{equation}\label{equrkprkp}r'_k\leq r_k<\beta\beta_1R_i\leq \beta\beta_1\diam W_i\leq R_k'.
\end{equation}
We need to show that there exists $\delta_1>0$ such that  $r'_m\geq \delta_1$. 

Combining (\ref{equpcetens}) with Koebe distortion theorem,  there exists $\alpha_0\in (0,1/2)$ such that 
\begin{equation}\label{equrpkrpkal}\frac{r'_k}{R'_k}>2\alpha_0,
\end{equation}
provided that $\xi_{a,k}$ is injective when restricted on $\rho_m\D$ and $g^k$ is injective on $2D$.
\par Set $\alpha_1:= \alpha_0^{l^N}/ (2\theta^{Nl^N})$. Combing Lemma \ref{pr2} (ii) which (\ref{equRk'bound}), shrink $\delta_0$ if necessary we can  choose $N_0$ large enough such that for $N_0\leq k\leq m-1$, we have 
\begin{equation}\label{equalphaoctrrp}\frac{\alpha_1^l}{\theta}>2 \left(C_2{R'_k}^{1/l}+\frac{L\rho_m}{R'_{k+1}}\right) .
\end{equation}
Since 
$$B'_k\subseteq_p \xi_{a,k}(\rho_m/2)\D,0) \text{ and } B(g^{k+1}(z),\rho_{\ast}(g, B'_k, g^{k}(z)))\subseteq_p g(B'_k, g^{k}(z)),$$
we have $B(g^{k+1}(z),\rho_{\ast}(g, B'_k, g^{k}(z)))\subseteq_p (g\circ \xi_k)((\rho_m/2)\D, 0).$
Hence $$\rho_{\ast}(g, B'_k, g^{k}(z))\leq \rho_*(g\circ \xi_k, (\rho_m/2)\D, 0).$$
By (\ref{equwidiam}) and (\ref{condiepplrhost}), we may apply (\ref{equghdown}) of Lemma \ref{perturb} and get 
\begin{equation}\label{4.14} 
	r'_{k+1}\geq \rho_{\ast}(g, B'_k, g^{k}(z))-L\rho_m.
\end{equation}
Hence we have
\begin{equation}\label{equcomrkprkp}
\begin{split}
\frac{r'_{k+1}}{R'_{k+1}}\geq&\frac{\rho_{\ast}(g, B'_k, g^{k}(z))}{\rho^\ast (g,g^{k}(D),g^{k}(z))}-\frac{L\rho_m}{R'_{k+1}}\\
\geq&\frac{\rho_{\ast}(g, B'_k, g^{k}(z))}{\rho^\ast (g,B(g^k(z),R_k'),g^{k}(z))}-\frac{L\rho_m}{R'_{k+1}}
\end{split}
\end{equation}
For $0\leq k\leq m-1$ such that $k\notin E_m$, by (\ref{equrkprkp}), (\ref{4.3}) of Lemma \ref{distor2}  and (\ref{equcomrkprkp}), we have
\begin{equation*}
	\frac{r'_{k+1}}{R'_{k+1}}\geq \frac{r'_k}{R'_k}-C_2(R'_k)^{1/l}-\frac{L\rho_m}{R'_{k+1}}.
\end{equation*}
If in addition $k\geq N_0$ and $r'_k/R'_k>\alpha_1$,  the above inequality and (\ref{equalphaoctrrp}) implies 
\begin{equation}\label{4.15}
	\log \frac{r'_{k+1}}{R'_{k+1}}\geq \log \frac{r'_k}{R'_k}-\frac{2}{\alpha_1}\left(C(R'_k)^{1/l}+\frac{L\rho_m}{R'_{k+1}}\right),
\end{equation}
here we use the inequality $\log(1-a)\geq -2a$ for $0<a\leq 1/2$. 
If  $k\in E_m$, by (\ref{equrkprkp}), (\ref{4.5}) of Lemma \ref{distor3}  and (\ref{equcomrkprkp}), we have
\begin{equation*}
	\frac{r'_{k+1}}{R'_{k+1}}\geq\frac{1}{\theta} \frac{(r'_k)^l}{(R'_k)^l}-\frac{L\rho_m}{R'_{k+1}}.
\end{equation*}
If in addition $k\geq N_0$ and $r'_k/R'_k>\alpha_1$,  by (\ref{equalphaoctrrp}) the above inequality implies 
\begin{equation}\label{4.16}
\log	\frac{r'_{k+1}}{R'_{k+1}}\geq l\log \frac{r'_k}{R'_k}-\frac{2\theta}{\alpha_1^l}\frac{L\rho_m}{R'_{k+1}}-\log\theta,
\end{equation}
again we are using $\log(1-a)\geq -2a$ for $0<a\leq 1/2$. 
By {\ref{equRk'bound}}, Lemma \ref{pr2} (ii) and the $\PCE(\la_0)$ condition, shrink $\delta_0$ if necessary, we choose $N_0$ large enough such that 
\begin{equation*}
l^N \sum_{k=N_0}^{m-1}\left( \frac{2}{\alpha_1}C{R'_k}^{1/l}+\frac{2\theta}{\alpha_1^l}\frac{L\rho_m}{R'_{k+1}}\right)\leq \log 2.
\end{equation*}

Since $\rho_m\to 0$ as $m\to \infty$, for $m$ sufficient large, by (\ref{equpcetens}) and the $\PCE(\la_0)$ condition,  
$\xi_{a,N_0}$ is injective when restricted on $\rho_m\D$ and $g^{N_0}$ is injective on $2D$. By (\ref{equrpkrpkal}), 
\begin{equation}\label{equindnz}\frac{r'_{N_0}}{R_{N_0}'}> 2\alpha_0>\alpha_1.
\end{equation}
We show that $\frac{r'_{p}}{R_{p}'}>\alpha_1$ for every $p=N_0+1,\dots, m$ by induction.
By (\ref{equindnz}) and the induction hypothesis, we assume that $\frac{r'_{p-1}}{R_{p-1}'}>\alpha_1$.
By (\ref{4.15}) and (\ref{4.16}),  we have
\begin{align*}
	\log \frac{r'_{p}}{R_{p}'}&\geq -l^N \sum_{k=N_0}^{p}\left( \frac{2}{\alpha_1}C{R'_k}^{1/l}+\frac{2\theta}{\alpha_1^l}\frac{L\rho_m}{R'_{k+1}}\right)+l^N \left(\log \frac{r'_{N_0}}{R'_{N_0}}-N\log\theta\right)
	\\&> l^N(\log \alpha_0-N\log\theta)-\log 2
	\\&=\log \alpha_1.
\end{align*}
Since $W_m=B(g^m(z), \delta_0)$, we have $R_m= \delta_0.$
By (\ref{equRk'bound}), we get 
$\frac{r'_m}{\delta_0}\geq \beta \frac{r'_m}{R'_m}\geq \beta\alpha_1.$
Set $\delta_1:=\beta\alpha_1\delta_0$, then $B(g^m(z),\delta_1)\subseteq h_m\left(\frac{1}{2}\D\right)$, which concludes the proof. 
\endproof


\subsection{Construction of invariant correspondences}

\subsubsection{Read maximal entropy measure from bifurcation measure}

We say that two sequences of positive real numbers $\rho_n, n\geq 0$ and $\rho_n', n\geq 0$ are \emph{equivalent} if $\log (\rho_n/\rho'_n), n\geq 0$ is bounded.

 Let $\mu_{f_0}$ be the maximal entropy measure of $f_0$.
For $\rho\in (0,1]$, let $[\rho]: \D\to \D$ be the map $t\mapsto \rho t$.

The following two results show that under the assumption of Theorem \ref{renor},
 $\mu_{f_0}$ can be read from $\mu_{f,a}$ via a suitable rescaling process. 
Moreover, the scales are determined by $\mu_{f,a}$ itself up to equivalence.  

 Recall that for a holomorphic family of rational maps $f:\D\times \P^1\to \D\times \P^1$  as in (\ref{family}) and for  a marked point $a$,  we let $\xi_{a,n}:\D\to \P^1(\C)$ be the map $\xi_{a,n}(t):=f_t^n(a(t))$. 
\begin{prop}\label{measure}
Let $f:\D\times \P^1\to \D\times \P^1$ be a holomorphic family of rational maps as in (\ref{family}) and $a$ be a marked point. 
Let $n_j, j\geq 0$ be an infinite subsequence of $n\geq 0.$
Let $\rho_{n_j}, j\geq 0$ be a sequence of positive real number tending to $0.$
Define $h_{n_j}:=\xi_{a,n_j}\circ [\rho_{n_j}].$
Assume that $h_{n_j}\to h$ locally uniformly. Then we have 
\begin{equation}\label{equmealimju}d^{n_j}[\rho_{n_j}]^*\mu_{f,a}\to h^\ast (\mu_{f_0})
\end{equation}
where the convergence is the weak convergence of measures. 
Moreover, if $a(0)\in \sJ(f_0)$ and $h$ is non-constant, then $0\in \supp\,h^\ast (\mu_{f_0}).$
\end{prop}
\begin{proof}
Define $H_{n_j}:\D\to \D\times \P^1(\C)$ by $H_n(t):=(\rho_n t,h_n(\rho_n t))$ and $H:\D\to \D\times \P^1(\C)$ by $H(t):=(0,h(\rho_n t))$.
Since $H_{n_j}=f^{n_j}(a)\circ[\rho_{n_j}]$, we have 
\begin{equation}\label{equhrhonj}H_{n_j}^\ast (T_f)=[\rho_{n_j}]^*((f^{n_j}(a))^*T_f)=d^{n_j}[\rho_{n_j}]^*\mu_{f,a}.
\end{equation}
Since $H_{n_j}\to H$ locally uniformly, $H_{n_j}^\ast (T_f)\to H^\ast (T_f)$. Since $H^\ast (T_f)=h^\ast (T_f\wedge [t=0])=h^*(\mu_{f_0})$, we get (\ref{equmealimju}) by (\ref{equhrhonj}). 
Assume that $a(0)\in \sJ(f_0)$ and $h$ is not constant. Since $h(0)=\lim\limits_{j\to \infty}f_0^{n_j}(a(0))\in \sJ(f_0),$ $0\in \supp\,h^\ast (\mu_{f_0}).$ This concludes the proof.
\end{proof}

\begin{pro}\label{proresunique}
Let $\mu$ be a Borel  measure on $\D$.
Let $\rho_n,\rho'_n, n\geq 0$ be two sequences of real numbers in $(0,1]$. Let $d_n, n\geq 0$ be a sequence of positive real numbers.
Let $\mu_1,\mu_2$ be Borel  measures on $\D$ having positive mass.
Assume that $\mu_1(\{0\})=\mu_2(\{0\})=0$.
If $d_n[\rho_n]^*\mu\to \mu_1$ and $d_n[\rho_n']^*\mu\to \mu_2$, then 
 $\rho_{n_j}, j\geq 0$ and $\rho'_{n_j}, j\geq 0$ are equivalent.  
\end{pro}
\proof
Assume by contradiction that it is not the case, then without loss of generality, by passing to a subsequence we may assume that
$r_n:=\frac{\rho'_{n}}{\rho_{n}}\to 0$ as $n\to \infty.$
Set $D_n:=\D(0,r_n).$ 
We have 
\begin{equation}\label{equmurhod}d_n[\rho_n]^*\mu(D_n)=d_n[\rho_n']^*\mu(\D).
\end{equation}
By the property of convergence of measures we  have
$$\liminf_{n\to \infty}d_n[\rho_n']^*\mu(\D)\geq \mu_2(\D)>0.$$
So for $n\gg 0$, 
$d_n[\rho_n']^*\mu(\D)\geq \mu_2(\D)/2.$
By (\ref{equmurhod}), for $n\gg 0$, 
$$d_n[\rho_n]^*\mu(\overline{D_n})\geq d_n[\rho_n]^*\mu(D_n)\geq \mu_2(\D)/2.$$
Since $r_n\to 0$ as $n\to 0,$ for every $r\in (0,1)$, we have 
$$\mu_1(\D(0,r))\geq \lim\sup_{n\to \infty}d_n[\rho_n]^*\mu(\overline{D_n})\geq \mu_2(\D)/2.$$
So $\mu_1(\{0\})\geq \mu_2(\D)/2$ which contradicts to our assumption. 
\endproof

\subsubsection{Asymptotic symmetry}
Let $X$ be a complex manifold and Let $\sH:=\{h_i, i\in A\}\subseteq \Hol(\D,X)$ be a family holomorphic maps from $\D$ to $X.$
Let $\lim\sH$ be the set of $h\in \Hol(\D,X)$, for which there is an infinity sequence of distinct $i_n\in A$ such that $h_{i_n}\to h$ as $n\to \infty.$
It is clear that $\lim\sH$ is closed in $\Hol(\D,X)$ and is contained in $\overline{\sH}.$
The family $\sH$ is normal if and only if $\lim\sH$ is compact. 
\begin{rem}
Note that $\Hol(\D,X)$ is a metric space. So by Lindel\"of property, a subset of $\Hol(\D,X)$ is compact if and only if it is sequentially compact.
\end{rem}
We say that $\sH$ is \emph{non-degenerate} if $\sH$ is normal and $\lim\sH$ does not contain any constant map\footnote{When $X=\P^1$, such a $\sH$ is called ``a non-trivial normal family" in \cite{Levin1990}.}.

\medskip

Let $g$ be a non-exceptional 
rational map of degree $d\geq 2.$ Let $\mu_{g}$ be the maximal entropy measure of $g$.
Let $\sH:=\{h_i, i\in A\}, \sH':=\{h_i', i\in A\}$ be two families of holomorphic maps from $\D$ to $\P^1.$
Set $\sH\times_A\sH':=\{h_i\times h_i': \D\to \P^1\times\P^1\}\subseteq \Hol(\D,\P^1\times \P^1).$ 
It is clear that $\sH\times_A\sH'$ is a closed subset of $\sH\times\sH'$.
Such a pair $(\sH, \sH')$ of families is called an \emph{asymptotic symmetry} of $g$ if the following conditions holds:
\begin{points}
\item  both $\{h_i, i\in A\}$ and $\{h_i', i\in A\}$ are non-degenerate;
\item $h_i(0), h_i'(0)\in \sJ(g);$
\item for every $\phi=h\times h'\in \lim (\sH\times_A\sH')$, we have $h^*\mu_g$ and $(h')^*\mu_g$ are proportional. 
\end{points}
We note that for an asymptotic symmetry $(\sH, \sH')$ and every $\phi=h\times h'\in \lim (\sH\times_A\sH')$, we have 
$$0\in h^{-1}(\sJ(g))=\supp\, h^*\mu_g=\supp\, (h')^*\mu_g=(h')^{-1}(\sJ(g)).$$

\medskip

Let $U$ be an open subset on $\P^1$.
Let $S(U)$ be the set of injective holomorphic maps $\sigma: U\to \P^1$ such that $\sigma^*\mu_g$ and $\mu_g|_U$ are nonzero and proportional.
The following theorem was proved in \cite[Theorem 1.7]{jixielocal2022}.
\begin{thm}\label{jxielocjr}For every $\sigma\in S(U)$, there is 
$\Gamma\in \Corr(\P^1)^g_*$ such that the image of $\id\times \sigma$ is contained in $\Gamma.$
\end{thm}

\begin{rem}
To prove our main result Theorem \ref{dao}, we only need the above theorem in the case where $g$ is $\TCE.$
This $\TCE$ case can be proved by  applying Dujardin-Favre-Gauthier's earlier results
\cite[Theorem A]{dujardin2022two} and \cite[Corollary 3.2]{dujardin2022two}. 
\end{rem}

Recall the following theorem of Levin \cite[Theorem 1]{Levin1990}.
\begin{thm}\label{thmlevin}
Every non-degenerate family $\sF\subseteq S(U)$ is finite. 
\end{thm}


\begin{thm}\label{thmasymsymclosedtosym}Let $(\sH, \sH')$ be an asymptotic symmetry of $g$. Then there is $\Gamma\in \Corr(\P^1)^g_*$ such that for every $\ep>0$,
there is a finite subset $E$ of $A$ such that for every $i\in A\setminus E$ and $t\in 1/2\D$, we have $$d((h_i(t),h'_i(t)), \Gamma)< \ep.$$
\end{thm}

\proof 
For every $\phi\in \Hol(\D,\P^1\times \P^1)$ and $\ep'>0$, let $N(\phi,\ep')$ be the open neighborhood of $\phi$ consisting of $\phi'\in \Hol(\D,\P^1\times \P^1)$
such that for every $t\in (9/10)\overline{\D}$, $d(\phi(t),\phi'(t))<\ep'.$

For every $\phi:=h\times h'\in \lim (\sH\times_A\sH')$, there is a point $t_{\phi}\in (1/2\D)\cap h^{-1}(\sJ(g))$ and $r_{\phi}<1/100$ such that both
$h$ and $h'$ are injective on $\D(t_{\phi}, 10r_{\phi}).$  There is $s_{\phi}>0$ such that $B_{\phi}:=B(h(t_{\phi}), s_{\phi})\subset\subset h(\D(t_{\phi}, r_{\phi})).$
There is $\ep_{\phi}>0$ such that for every $\psi=l\times l'\in N(\phi,\ep_{\phi})$, we have
\begin{points}
\item both $l$ and $l'$ are injective on $\D(t_{\phi}, 9r_{\phi})$;
\item $B_{\phi}\subset\subset l(\D(t_{\phi}, r_{\phi}))$;
\item $|d(l'\circ (l|_{\D(t_{\phi}, r_{\phi})})^{-1})(t_{\phi})|\in \left(9/10\frac{|dh(t_{\phi})|}{|dh(t_{\phi})|}, 11/10\frac{|dh(t_{\phi})|}{|dh(t_{\phi})|}\right).$
\end{points}
Then for every $\psi=l\times l'\in N(\phi,\ep_{\phi})\cap \lim (\sH\times_A\sH')$,
$\sigma_{\psi}:=l'\circ (l|_{\D(t_{\phi}, r_{\phi})})^{-1}|_{B_{\phi}}: B_{\phi}\to \P^1$ is in $S(B_{\phi}).$
Moreover, by (iii) and Koebe distortion theorem, the family $\{\sigma_{\psi}, \psi\in N(\phi,\ep_{\phi})\cap \lim (\sH\times_A\sH')\}$
is non-degenerate. Then it is finite by Theorem \ref{thmlevin}.
By Theorem \ref{jxielocjr}, there is $\Gamma_{\phi}\in \Corr(\P^1)^g_*$ such that 
for every $\psi\in N(\phi,\ep_{\phi})\cap \lim (\sH\times_A\sH')\}$, the image of 
$\id\times \sigma_{\psi}$ is contained in $\Gamma_{\phi}.$
Hence the image of $\psi$ is contained in $\Gamma_{\phi}.$

Since $\lim (\sH\times_A\sH')$ is compact, there is a finite set $F\subseteq (\sH\times_A\sH')$ such that 
$\lim(\sH\times_A\sH')\subseteq \cup_{\phi\in F}N(\phi,\ep_{\phi}).$
Then for every $\psi\in \lim (\sH\times_A\sH')\}$, the image of 
$\psi$ is contained in $\Gamma:=\cup_{\phi}\Gamma_{\phi}\in  \Corr(\P^1)^g_*.$
Set $W:=\cup_{\phi\in \lim (\sH\times_A\sH')}N(\phi, \ep)$. It is an open neighborhood of $\lim (\sH\times_A\sH').$
Then $E:=\{i\in A|\,\, h_i\times h'_i\not\in W\}$ is finite.
For every $i\in A$, there is $\phi\in \lim (\sH\times_A\sH')$ such that $h_i\times h'_i\in N(\phi, \ep).$
Since the image of $\psi$ is contained in $\Gamma$ and $1/2\D\subset\subset (9/10)\D$, for every $t\in 1/2\D$, we have 
$$d((h_i(t),h'_i(t)), \Gamma)\leq d((h_i(t),h'_i(t)), \phi(t))< \ep.$$
This concludes the proof.
\endproof

\begin{prop}\label{Julia}
Let $f:\D\times \P^1\to \D\times \P^1$ be a holomorphic family of rational maps as in (\ref{family}).
Let $a, b$ be two marked points. 
Assume that for both $a$ and $b$, the parameter
$0\in \D$ satisfies
\begin{points}
	\item $\PCE(\la_0)$ for some $\la_0>1$;
	\item $\PR(s)$ for some $s>0$;
	\item $f_{0}$ is $\TCE(\la)$ for some $\la>1$;
	\end{points} 
Assume further that $\mu_{f,a}$ and $\mu_{f,b}$ are proportional and 
$f_0$ is not exceptional.
Then there is $\Gamma_0\in \Corr(\P^1)^{f_0}_*$ such that the $\FS(\Gamma_0)$ condition does not hold.
\end{prop}
\begin{proof}
By Theorem \ref{DGV}, 
the $\PCE(\la_0)$ condition implies the marked Collet-Eckmann condition for $a(0)$ and $b(0).$
In particular, both $a(0)$ and $b(0)$ are contained in $\sJ(g).$

By Theorem \ref{renor}, there are 
 $A_a, A_b\subseteq \Z_{\geq 0}$ with $\underline{d}(A_a)>9/10$ and $\underline{d}(A_b)>9/10$ and  two sequences of positive numbers
$\rho_{a,n}, n\in A_a,$ and $\rho_{b,n}, n\in A_b$ both tending to $0$ as $n\to \infty$ such that   $h_{a,n}:\D\to \P^1(\C), n\in A_a$ and $h_{b,n}:\D\to \P^1(\C), n\in A_b$ as in Theorem \ref{renor} are non-degenerate.
Set $A:=A_a\cap A_b.$ Then $\underline{d}(A)>0.8.$ 
The families $\sH_a:=\{h_{a,n}, n\in A\}$ and $\sH_b:=\{h_{b,n}, n\in A\}$ are non-degenerate.

\medskip

We claim that $\rho_{a,n}, n\in A$ and $\rho_{b,n}, n\in A$ are equivalent. Otherwise, we may assume that there is a sequence $n_j\in A$ tending to $+\infty$
such that $\rho_{a,n_j}/\rho_{b,n_j}\to +\infty.$ After taking subsequence, we may assume that $h_{a,n_j}\to h_a$ and $h_{b,n_j}\to h_b$, where $h_a$ and $h_b$ are non-constant holomorphic maps. 
Since  $a(0),b(0)\in \sJ(f_0),$ 
by Proposition \ref{measure}, we have 
\begin{equation}\label{equabmeacon}
d^{n_j}[\rho_{a,n_j}]^*\mu_{f,a}\to h_a^\ast \mu_{f_0} \text{ and } d^{n_j}[\rho_{b,n_j}]^*\mu_{f,b}\to h_b^\ast \mu_{f_0}
\end{equation}
and $0$ is contained in both the supports of $h_a^\ast (\mu_{f_0})$ and $h_b^\ast (\mu_{f_0}).$
Since $\mu_{f_0}$ has continuous potential, it does not have atoms.
Since $\mu_{f,a}=c\mu_{f,b}$, by Proposition \ref{proresunique}, $\rho_{a,n_j}, j\geq 0$ and $\rho_{b,n_j}, j\geq 0$ are equivalent, which is a contradiction.
Hence our claim holds.
After multiplying $\rho_{a,n}$ by a constant in $(0,1]$, we may assume that $\rho_{a,n}\leq \rho_{b,n}$ for every $n\in A.$
There is $q'\in (0,1]$ such that $q_n:=\rho_{a,n}/\rho_{b,n}\in [q',1]$ for every $n\in A.$ 
After replacing $\rho_{b,n}$ by $\rho_{a,n}=q_n\rho_{b,n}$ and 
$h_{b,n}$ by $h_{b,n}\circ [q_n]$, we may assume that $\rho_{a,n}=\rho_{b,n}$ for every $n\in A.$
Set $\rho_n:=\rho_{a,n}=\rho_{b,n}.$
Then for any sequence $n_j\in A$ tending to $+\infty$ satisfying 
$h_{a,n_j}\to h_a$ and $h_{b,n_j}\to h_b$,   (\ref{equabmeacon}) becomes
\begin{equation}\label{equabmeaconp}
d^{n_j}[\rho_{n_j}]^*\mu_{f,a}\to h_a^\ast \mu_{f_0} \text{ and } d^{n_j}[\rho_{n_j}]^*\mu_{f,b}\to h_b^\ast \mu_{f_0}
\end{equation}
Since $\mu_{f,a}$ and $\mu_{f,b}$ are proportional and $0$ is contained in both of $\supp\, h_a^\ast (\mu_{f_0})$ and $\supp\,  h_b^\ast (\mu_{f_0}),$
$h_a^\ast (\mu_{f_0})$ and $h_b^\ast (\mu_{f_0})$ are non-zero and proportional. 
This implies that the pair of families $(\sH_a, \sH_b)$ is an asymptotic symmetry.
By Theorem \ref{thmasymsymclosedtosym}, there is $\Gamma_0\in \Corr(\P^1)^{f_0}_*$ such that for every $\ep>0$,
there is a finite subset $E_{\ep}$ of $A$ such that for every $n\in A\setminus E_{\ep}$, we have $$d((h_{a,n}(0),h_{b,n}(0)), \Gamma_0)< \ep.$$
Hence, for every $\ep>0$, 
$$\overline{d}(\{n\geq 0|\,\, d((h_{a,n}(0),h_{b,n}(0)), \Gamma_0)\geq \ep\})\leq 1-\underline{d}(A)< 1/5.$$
So the $\FS(\Gamma_0)$ condition does not hold, which concludes the proof.
\end{proof}

\section{Proof of the main theorems}\label{6}


\proof[Proof of Theorem \ref{thmbogodao}]
The direction (ii) implies (i) was proved by DeMarco \cite[Section 6.4]{demarco2016bifurcations}. 
The direction (i) implies (iii) is trivial. 
We only need to show that (iii) implies (ii).

We may assume that $\phi_f(\La)$ is not contained in the locus of flexible Latt\`es maps, otherwise Theorem \ref{thmbogodao} trivially holds.
As $f$ is defined over $\overline{\Q},$ there is a variety $\La_0$ over $\overline{\Q}$ and a morphism $F:\La_0\times \P^1\to \P^1$ over $\overline{\Q}$ such that $\La=\La_0\otimes_{\overline{\Q}}\C$ and $f=F\otimes_{\overline{\Q}}\C$. After replacing $\La_0$ by a finite ramification cover on it over $\overline{\Q}$, we may assume that $F$ has $2d-2$ marked critical points $a_{i}, i=1,\dots,2d-2$ counted with multiplicity. Let $c_i$ be the base change of  $a_i.$ 
Then we get marked critical points $(c_i)_{1\leq i\leq 2d-2}.$ 

\medskip

Since  $f$ is not isotrivial and $\phi_f(\La)$ is not contained in the locus of flexible Latt\`es maps, $\mu_{\bif}\neq 0$.
By Corollary \ref{corequi}, for every $1\leq i\leq 2d-2$, $\mu_{f,c_i}$ is proportional to $\mu_{\bif}$. Moreover, by Theorem \ref{stable}, $\mu_{f,c_i}\neq 0$ if and only if $c_i$ is active.
Let $c_1$, $c_2$ be two marked critical points, we need to show they are dynamically related. If one of $c_1$, $c_2$ is preperiodic, then $c_1$ and $c_2$ are dynamically related holds trivially. So we only need to consider the case that  $c_1$ and $c_2$ are active. Assume for the sake of contradiction that $c_1$ and $c_2$ are active but not dynamically related.

\medskip
By Corollary \ref{corinvariantas} we have that
 
 \begin{points}
\item[(a)]for $\mu_{\bif}$-a.e. point $t\in\La(\C)$,  for every  $\Gamma_t\in \Corr(\P^1)^{f_t}_*$,  the pair $c_1(t), c_2(t)$ satisfies the $\AS(\Gamma_t)$ condition for $f_t$.
\end{points}
Let $1<\la<d^{1/2}$ and $s>1/2.$
By Proposition \ref{nuh} we have  
\begin{points}
\item[(b)] for $\mu_{\bif}$-a.e. $t\in \La(\C)$, $t$ satisfies $\CE^*(\la)$, $\PCE(\la)$ and $\PR(s)$ for all $c_i$ that is active.
\end{points}

\begin{defi}\label{ce}
A rational map $g$ of degree at least $2$ is called {\em Collet-Eckmann} $\CE(\la)$ for some $\la>1$ if:
\begin{points}
\item There exists $C>0$ such that  for every critical point $c\in\sJ(g)$, there exists $N>0$ such that $|dg^n(g^N(c))|\geq C\la^n$ for every $n\geq 1$;
\item $g$ has no parabolic cycle. 
\end{points}
\end{defi}
\medskip
\par The following proposition is  useful.  A precise statement can be found in   \cite[Main Theorem]{przytycki2003equivalence}.
\begin{prop}[Przytycki-Rohde \cite{przytycki1998porosity}]\label{cetce}
$\CE(\la_0)$  implies $\TCE(\la)$ for every $1<\la<\la_0$.  
\end{prop}
\par We have the  following two results.
\begin{prop}\label{generic}
There exists  $\la>1$ such that 
\begin{points}
\item[\rm{(c)}] for $\mu_{\bif}-$a.e. point $t$, $f_t$ is $\CE(\la)$.
\end{points}
\end{prop}
\begin{proof}
By Theorem \ref{stable},  a marked critical point $c$ is either preperiodic or satisfies $\mu_{f,c}>0$.  Let $I\subseteq \left\{1,2,\dots,2d-2\right\}$ be the index such that $c_i$ is preperiodic when $i\in I$.

By (b), for every $i\in \left\{1,2,\dots,2d-2\right\}\setminus I$, $c_i(t)$ satisfies  (i) in Definition \ref{ce} for  $\mu_{\bif}$-a.e. point $t$.
This implies that for $\mu_{\bif}$-a.e. point $t$, the critical points are either contained in the Julia set or preperiodic.
Note that every parabolic basin contains a critical point,  and this critical point can not be preperiodic. 
This implies that $\mu_{\bif}$-a.e. point $t$ satisfies (ii) in Definition \ref{ce} i.e. $f_t$ has no parabolic cycle.

\par For each fixed $i\in I$, we let  $p_i$ be the cycle which $c_i$ is preperiodic to. We only need to show the following claim:
$\mu_{\bif}$-a.e. point $t\in \La$ satisfies the following property: if $c_i(t)\in \sJ(f_t)$, then $c_i(t)$ is preperiodic to a repelling cycle.  
\par Let $\la_{p_i}:\La\to \C$ be the holomorphic map defined by the multiplier of the cycle $p_i, i\in I$.  There are two cases.
\par {\em Case 1}: $\la_{p_i}$ is a constant map. 
As $f$ is defined over $\overline{\Q}$, $\la_{p_i}\in \overline{\Q}.$ If $|\la_{p_i}|\neq 1$, then the claim holds. If $|\la_{p_i}|=1$, since for $\mu_{\bif}$-a.e. point $t$, $p_i(t)$ is not a parabolic cycle, $\la_{p_i}$ is not a root of unity.
By Siegel's linearization theorem \cite[Theorem 11.4]{milnor2011dynamics} and Baker's theorem (c.f. \cite[Theorem 3.1]{Baker2022}),  $p_i(t)$ is a Siegel cycle. This implies our claim. 

\par {\em Case 2}: $\la_{p_i}$ is not a constant map. If the claim is not true for a parameter $t$,  then $c_i(t)$ is preperiodic to a parabolic cycle or a Cremer cycle.  
Since $\la_{p_i}$ is not a constant map, $\{t\in \La(\C)|\,\,  \la_{p_i} \text{ is a root of unity}\}$ is countable.
Recall that by Siegel's linearization theorem \cite[Theorem 11.4]{milnor2011dynamics},  if $p_i(t)$ is a Cremer cycle, then $\la_{p_i}(t)=e^{2\pi i \alpha}$ , where $\alpha$ is a Liouville number.   Recall that the set of Liouville numbers has Hausdorff dimension $0$  \cite[Lemma C.7]{milnor2011dynamics}. Let $\mathcal{CP}\subseteq \La$ be the parameters in $\La$ such that $p_i(t)$ is Cremer or parabolic. Then the Hausdorff  dimension of  $\mathcal{CP}$  is $0$. On the other hand since $\mu_{\bif}$ has H\"older continuous potential \cite[Lemma 1.1]{dinh2010dynamics}, the Hausdorff dimension of $\mu_{\bif}$ is strictly positive \cite[Theorem 1.7.3]{sibony1999dynamique}. This implies our claim.
\end{proof}

Since $\phi_f(\La)$ is not contained in the locus of flexible Latt\`es maps and not a point, 
the parameters $t$ such that $f_t$ is exceptional are finite.  
Since $\mu_{\bif}$ does not have atoms,  we have 
\begin{points}
\item[(d)] for $\mu_{\bif}$-a.e. point $t\in \La(\C)$,$f_t$ is not exceptional.
\end{points}

Combining (b), (c), (d) and  Proposition \ref{Julia}, we get that 
\begin{points}
\item[(e)] for $\mu_{\bif}$-a.e. point $t\in \La(\C)$, there is $\Gamma_t\in \Corr(\P^1)^{f_t}_*$ such that the $\FS(\Gamma_0)$ condition does not hold.
\end{points}
Since (e) contradicts (a), we finish the proof of Theorem \ref{dao}.
\endproof

We now deduce Theorem \ref{dao} from Theorem \ref{thmbogodao}.
\proof[Proof of Theorem \ref{dao}]
By Theorem \ref{thmbogodao}, we only need to reduce to the case where $f$ is defined over $\overline{\Q}.$
Set $\psi:=\Psi\circ \phi_f: \La\to \sM_d.$
Let $B$ be the Zariski closure of $\psi(\La)$ in $\sM_d.$ 
Since $f$ is not isotrivial,
$B$ is a curve and the map $\psi: \La\to B$ is quasi-finite.
We may assume that $\phi_f(\La)$ is not contained in the locus of flexible Latt\`es maps, otherwise Theorem \ref{dao} trivially holds.
 Since there are infinitely many PCF parameters $t\in \La$, by Thurston's rigidity of PCF maps \cite{Douady1993}, $B\cap \sM_d(\overline{\Q})$ is infinite. Hence $B$ is defined over $\overline{\Q}$. Since $Z:=\Psi^{-1}(B)$ is defined over $\overline{\Q}$, there is a smooth curve $\La_1$ and a morphism $\phi_1: \La_1\to Z$ defined over $\overline{\Q}$ such that the morphism $\psi_1:=\Psi\circ \phi_1: \La_1\to B$ is dominant. Let $g: \La_1\times \P^1\to  \P^1$ be the family of rational maps induced by $\phi_1.$ It is clear that Theorem \ref{dao} holds for $f$ if and only if it holds for $g$, this concludes the proof.
\endproof

\end{document}